\documentclass[reqno,a4paper]{article}
\usepackage{etex}
\usepackage[T1]{fontenc}
\usepackage[utf8]{inputenc}
\usepackage{lmodern}

\usepackage[style=alphabetic,firstinits=true,backref,backend=biber,isbn=false,url=false,maxbibnames=99]{biblatex}

\renewbibmacro{in:}{}
\newbibmacro{string+doi}[1]{\iffieldundef{doi}{\iffieldundef{url}{#1}{\href{\thefield{url}}{#1}}}{\href{http://dx.doi.org/\thefield{doi}}{#1}}}
\DeclareFieldFormat*{title}{\usebibmacro{string+doi}{\emph{#1}}}
\DefineBibliographyStrings{english}{%
  backrefpage = {$\uparrow$~p{.}},
  backrefpages = {$\uparrow$~pp{.}},
}

\usepackage{amssymb,amsmath,amsthm}
\usepackage{multirow}
\usepackage{units}
\usepackage{graphicx}
\usepackage[all]{xy}
\usepackage{tikz}
\usetikzlibrary{arrows,calc,positioning,decorations.pathreplacing}
\usepackage{relsize}
\usepackage{mhsetup}
\usepackage{mathtools}
\usepackage{stmaryrd}
\usepackage{tocloft}
\usepackage{paralist} 
\usepackage[left=3cm,right=3cm,top=3cm,bottom=3cm]{geometry} 
\usepackage[colorlinks,citecolor=blue,linkcolor=blue,urlcolor=blue,filecolor=blue,bookmarksopen,bookmarksnumbered,bookmarksdepth=4,bookmarksopenlevel=2,breaklinks=true]{hyperref}
\usepackage{footnotebackref}

\definecolor{lightblue}{rgb}{0.8,0.8,1}

\numberwithin{equation}{section}
\numberwithin{figure}{section}

\newtheoremstyle{italicised}
        {\topsep}{\topsep}  
        {\itshape}  
        {}  
        {\bfseries}  
        {}  
        {1ex}  
        {}  
\theoremstyle{italicised}
\newtheorem{thm}{Theorem}[section]
\newtheorem{lem}[thm]{Lemma}
\newtheorem{prop}[thm]{Proposition}
\newtheorem{coro}[thm]{Corollary}

\newtheorem{fact}[thm]{Fact}

\newtheoremstyle{upright}
        {\topsep}{\topsep}  
        {\upshape}  
        {}  
        {\bfseries}  
        {}  
        {1ex}  
        {}  
\theoremstyle{upright}
\newtheorem{defn}[thm]{Definition}
\newtheorem{rmk}[thm]{Remark}
\newtheorem{eg}[thm]{Example}

\newtheoremstyle{italicised-restate}
        {\topsep}{\topsep}  
        {\itshape}  
        {}  
        {\bfseries}  
        {}  
        {1ex}  
        {\thmname{#1}\thmnote{ \bfseries #3}}  
\theoremstyle{italicised-restate}

\makeatletter
\renewcommand*{\@seccntformat}[1]{\upshape\csname the#1\endcsname.\hspace{1ex}}
\renewcommand*{\section}{\@startsection{section}{1}{\z@}%
  {2.5ex \@plus 1ex \@minus 0.2ex}%
  {1.5ex \@plus 0.2ex}%
  {\large\bfseries}}
\renewcommand*{\subsection}{\@startsection{subsection}{2}{\z@}%
  {2.5ex \@plus 1ex \@minus 0.2ex}%
  {-1.5ex \@plus -0.2ex}%
  {\normalfont\normalsize\bfseries}}
\renewcommand*{\subsubsection}{\@startsection{subsubsectionthat}{3}{\z@}%
  {2.5ex \@plus 1ex \@minus 0.2ex}%
  {-1.5ex \@plus -0.2ex}%
  {\normalfont\normalsize\bfseries}}
\renewcommand*{\paragraph}{\@startsection{paragraph}{4}{\z@}%
  {2.5ex \@plus 1ex \@minus 0.2ex}%
  {-1.5ex \@plus -0.2ex}%
  {\normalfont\normalsize\bfseries}}
\renewcommand*{\subparagraph}{\@startsection{subparagraph}{5}{\z@}%
  {2.5ex \@plus 1ex \@minus 0.2ex}%
  {-1.5ex \@plus -0.2ex}%
  {\normalfont\normalsize\slshape}}
\makeatother

\newcommand{\qedsquareb}{\raisebox{-0.2mm}{\tikz[x=1mm,y=1mm]{\draw[line cap=rect] (0,0)--(2,0)--(2,2); \draw[very thick, line cap=butt] (2,2)--(0,2)--(0,0);}}}
\renewcommand{\qedsymbol}{\qedsquareb}

\newcommand{\cA}{\mathcal{A}}
\newcommand{\cB}{\mathcal{B}}
\newcommand{\cC}{\mathcal{C}}
\newcommand{\cD}{\mathcal{D}}

\newcommand{\cI}{\mathcal{I}}
\newcommand{\cJ}{\mathcal{J}}

\newcommand{\cM}{\mathcal{M}}
\newcommand{\cN}{\mathcal{N}}

\newcommand{\cP}{\mathcal{P}}

\newcommand{\cS}{\mathcal{S}}

\newcommand{\bC}{\mathbb{C}}
\newcommand{\bD}{\mathbb{D}}

\newcommand{\bN}{\mathbb{N}}

\newcommand{\bR}{\mathbb{R}}

\newcommand{\bZ}{\mathbb{Z}}

\newcommand{\degree}{\ensuremath{\mathrm{deg}}}

\newcommand{\undm}{\ensuremath{\underline{m}}}
\newcommand{\undn}{\ensuremath{\underline{n}}}

\renewcommand{\geq}{\geqslant}
\renewcommand{\leq}{\leqslant}
\renewcommand{\footnoterule}{%
  \kern -3pt
  \hrule width \textwidth height 0.4pt
  \kern 2.6pt
}

\newcommand{\incl}[3][right]%
{%
\draw[<-,>=#1 hook] #2 to ($ #2!0.5!#3 $);
\draw[->,>=stealth'] ($ #2!0.5!#3 $) to #3;%
}
\newcommand{\inclusion}[5][right]%
{%
\draw[<-,>=#1 hook] #4 to ($ #4!0.5!#5 $) node[#2,font=\small]{#3};
\draw[->,>=stealth'] ($ #4!0.5!#5 $) to #5;%
}

\newcommand{\crbar}{\ensuremath{\,\overline{\! \mathrm{cr}\!}\,}}
\newcommand{\htbar}{\ensuremath{\,\overline{\raisebox{0pt}[1.3ex][0pt]{\ensuremath{\! \mathrm{ht}\!}}}\,}}

\newcommand{\ab}{\ensuremath{\mathsf{Ab}}}

\newcommand{\Se}{\ensuremath{\mathsf{Se}^{\mathsf{fin}}}}

\newcommand{\mfdc}{\ensuremath{\mathsf{Mfd}_{\mathsf{c}}}}
\newcommand{\topo}{\ensuremath{\mathsf{Top}_{\circ}}}

\newcommand{\cats}{\ensuremath{\mathsf{Cat}_{\mathsf{s}}}}
\newcommand{\catst}{\ensuremath{\mathsf{Cat}_{\mathsf{st}}}}
\newcommand{\cati}{\ensuremath{\mathsf{Cat}_{\mathcal{I}}}}
\newcommand{\cBf}{\ensuremath{\mathcal{B}_{\mathsf{f}}}}
\newcommand{\placeholder}{\ensuremath{\mathsf{x}}}
\newcommand{\Monini}{\ensuremath{\mathcal{M}\mathrm{on}_{\mathit{ini}}}}
\newcommand{\Sn}{\ensuremath{\mathcal{S}\mathrm{n}}}
\newcommand{\I}[1]{\ensuremath{\mathcal{I}_{#1}}}
\newcommand{\J}[1]{\ensuremath{\bar{\mathcal{I}}_{#1}}}
\newcommand{\K}[1]{\ensuremath{\bar{\mathcal{I}}_{#1}^{\mathsf{op}}}}

\newcommand{\cf}{\textit{cf}.\ }
\newcommand{\Cf}{\textit{Cf}.\ }
\newcommand{\firsttype}{\ensuremath{\dagger}}
\newcommand{\secondtype}{\ensuremath{\ddagger}}
\newenvironment{itemizeb}%
{\begin{compactitem}

}%
{\end{compactitem}}

\begin{document}
\title{\Large\bfseries A comparison of twisted coefficient systems\vspace{-1ex}}
\author{\small Martin Palmer\quad $/\!\!/$\quad 23 February 2019\vspace{-1ex}}
\date{}
\maketitle
{
\makeatletter
\renewcommand*{\BHFN@OldMakefntext}{}
\makeatother
\footnotetext{\textit{Address}: Mathematisches Institut der Universit{\"a}t Bonn, Endenicher Allee 60, 53115 Bonn, Germany}
\footnotetext{\textit{Email address}: \textsf{palmer@math.uni-bonn.de}}
}
\begin{abstract}
In the first part of this note, we review and compare various instances of the notion of \emph{twisted coefficient system}, a.k.a.\ \emph{polynomial functor}, appearing in the literature. This notion hinges on how one defines the \emph{degree} of a functor from $\cC$ to an abelian category, for various different structures on $\cC$. In the second part, we focus on twisted coefficient systems defined on \emph{partial braid categories}, and explain a functorial framework for this setting.
\end{abstract}
\tableofcontents


\section{Introduction}\label{sec:introduction}

This note is concerned with \emph{twisted coefficient systems}, by which we mean simply functors from a category $\cC$ equipped with a certain structure to an abelian category $\cA$. The structure on $\cC$ depends on the precise situation that one wishes to study, and is used to define a notion of \emph{degree} for any functor $\cC \to \cA$. The main goal of this note is to compare different structures on the source category $\cC$, and the resulting notions of degree.

Twisted coefficient systems of \emph{finite degree}, also known as \emph{polynomial functors}, are often used to study the homology of interesting spaces or groups (or other abelian invariants, such as the filtration quotients in the lower central series of a group), for example automorphism groups of free groups and congruence groups (\cf \cite[\S 5]{DjamentVespa2019FoncteursFaiblementPolynomiaux}). Indeed, polynomial functors first appeared in the paper~\cite{EilenbergMac1954groupsHnII} of Eilenberg and MacLane (see \S 9), where they were used to compute the homology, in a certain range of degrees, of the Eilenberg-MacLane spaces $K(A,n)$ for $n\geq 2$. This involves assembling the objects of interest (e.g.\ the homology of congruence subgroups) into a polynomial functor, and leveraging the fact that it has finite degree to study them.

On the other hand, polynomial functors may also appear as the \emph{coefficients} in homology groups: one may also be interested in the homology of a family of spaces or groups, with respect to a corresponding family of local coefficient systems \emph{that assemble into a polynomial functor} -- it is in this guise that polynomial functors are more commonly referred to as \emph{twisted coefficient systems} (of finite degree). Many families of groups or spaces are known to be \emph{homologially stable} with respect to (appropriately-defined) finite-degree twisted coefficient systems, including the symmetric groups, braid groups, configuration spaces, general linear groups, automorphism groups of right-angled Artin groups and mapping class groups of surfaces and of $3$-manifolds.\footnote{Symmetric groups: \cite{Betley2002Twistedhomologyof}; braid groups: \cite{ChurchFarb2013Representationtheoryand,Randal-WilliamsWahl2017Homologicalstabilityautomorphism,Palmer2018Twistedhomologicalstability}; configuration spaces: \cite{Palmer2018Twistedhomologicalstability}; general linear groups: \cite{Dwyer1980Twistedhomologicalstability,Kallen1980Homologystabilitylinear}; automorphism groups of free groups: \cite{Randal-WilliamsWahl2017Homologicalstabilityautomorphism}; automorphism groups of right-angled Artin groups: \cite{GandiniWahl2016Homologicalstabilityautomorphism}; mapping class groups of surfaces: \cite{Ivanov1993homologystabilityTeichmuller, CohenMadsen2009Surfacesinbackground,Boldsen2012Improvedhomologicalstability,Randal-WilliamsWahl2017Homologicalstabilityautomorphism}; mapping class groups of $3$-manifolds: \cite{Randal-WilliamsWahl2017Homologicalstabilityautomorphism}. Note that these are references for the proofs of \emph{twisted} homological stability; in many cases, homological stability with untwisted coefficients was known much earlier.}

\paragraph*{Cross-effects vs.\ endofunctors.}

There are two common approaches to defining the \emph{degree} of a functor $\cC \to \cA$. One approach uses certain structure on $\cC$ to define the \emph{cross-effects} of a given functor, which consists of an $\bN$-graded set of objects of $\cA$, and the degree is then determined by the vanishing of these objects. The idea is that the cross-effects encode information about the functor; if they vanish above a certain level, then this information is concentrated in an essentially finite amount of data, which makes it possible to prove certain things about the given functor (such as homological stability with coefficients in this functor). For example, this is the approach taken in the paper \cite{Palmer2018Twistedhomologicalstability}, from which this note arose, to prove homological stability for configuration spaces with finite-degree twisted coefficients.

The second approach is a recursive definition, depending on the choice of an endofunctor $s$ of $\cC$ and a natural transformation $\mathrm{id} \to s$. This allows one to prove things about functors of finite degree (in this sense) by induction on the degree. For example, the main theorem of \cite{Randal-WilliamsWahl2017Homologicalstabilityautomorphism}, providing a ``machine'' for proving twisted homological stability for families of groups, is an inductive proof of this kind, and the degree of their twisted coefficient systems is defined recursively.

\paragraph*{Degree and height.}

In this note, we will use the word \emph{degree} to refer to a definition of the second kind, i.e., a recursive definition, and we will use the word \emph{height} to refer to a definition of the first kind, given by the vanishing of certain \emph{cross-effects} above a certain ``height''.

In sections \ref{sec:inductive-degree} and \ref{sec:cross-effects} (respectively) we compare various notions of \emph{degree} and \emph{height} appearing in the literature. We will focus on comparisons between \cite{DjamentVespa2019FoncteursFaiblementPolynomiaux}, \cite{Randal-WilliamsWahl2017Homologicalstabilityautomorphism}, \cite{Ivanov1993homologystabilityTeichmuller}, \cite{CohenMadsen2009Surfacesinbackground}, \cite{Boldsen2012Improvedhomologicalstability} and \cite{Palmer2018Twistedhomologicalstability} for the degree, and between \cite{DjamentVespa2019FoncteursFaiblementPolynomiaux}, \cite{HartlVespa2011Quadraticfunctorspointed}, \cite{HartlPirashviliVespa2015Polynomialfunctorsalgebras}, \cite{CollinetDjamentGriffin2013Stabilitehomologiquepour} and \cite{Palmer2018Twistedhomologicalstability} for the height. We do not pursue here the relationship between the height and the degree of a functor $\cC \to \cA$ (when both are defined) in the greatest possible generality (although this is discussed in certain special cases, see Remark \ref{rmk:summary} for a summary). Rather, we focus on comparing (and unifying) \emph{with each other} the various different notions of \emph{degree} appearing in the literature, and similarly for the various different notions of \emph{height} in the literature.

\paragraph*{Historical remarks.}

The notion of what we term the \emph{height} of a functor was first introduced by Eilenberg and MacLane (who used the name \emph{degree}) in \cite[\S 9]{EilenbergMac1954groupsHnII}, where it was used to compute the integral homology of Eilenberg-MacLane spaces in a range of degrees. Somewhat later, it was used by Dwyer~\cite{Dwyer1980Twistedhomologicalstability} to formulate and prove a twisted homological stability theorem for general linear groups. The \emph{height} of a functor also appears in \cite{Pirashvili2000DoldKantype} (see \S 2.3) and was used in \cite{Betley2002Twistedhomologyof} to prove a twisted homological stability theorem for the symmetric groups. More recently, it also appears in \cite{HartlVespa2011Quadraticfunctorspointed}, \cite{HartlPirashviliVespa2015Polynomialfunctorsalgebras}, \cite{CollinetDjamentGriffin2013Stabilitehomologiquepour} (in which it is used to prove a homological stability theorem for automorphism groups of free products of groups) and \cite{DjamentVespa2019FoncteursFaiblementPolynomiaux} (which also introduces the notion of \emph{weak} polynomial functors, in contrast to \emph{strong} polynomial functors -- see also Definition \ref{def:degree-general} below, which is inspired by their definition).

The notion of what we term the \emph{degree} of a functor appeared first (as far as the author is aware) implicitly in the work of Dwyer~\cite{Dwyer1980Twistedhomologicalstability} and (slightly later) explicitly in the work of van der Kallen~\cite{Kallen1980Homologystabilitylinear} (see \S 5.5).\footnote{Dwyer~\cite{Dwyer1980Twistedhomologicalstability} explicitly defines a \emph{height}-like notion (at the beginning of \S 3), but there is a \emph{degree}-like notion implicit in his work (\cf Theorem 2.2 and the proof of Lemma 3.1). Van der Kallen~\cite{Kallen1980Homologystabilitylinear}, on the other hand, uses techniques similar to those of \cite{Dwyer1980Twistedhomologicalstability}, but explicitly uses a \emph{degree}-like notion (see \S 5.5), and remarks that functors of finite degree (in his sense) can be obtained from functors of finite degree in the sense of \cite{Dwyer1980Twistedhomologicalstability}.} This notion was also used by Ivanov~\cite{Ivanov1993homologystabilityTeichmuller} to formulate and prove a twisted homological stability theorem for mapping class groups, and analogous (not quite identical) definitions were used also by \cite{CohenMadsen2009Surfacesinbackground} and \cite{Boldsen2012Improvedhomologicalstability} in similar contexts (\cf \S \ref{para:degree-mcg}). The notion also appears in the work of Djament and Vespa~\cite{DjamentVespa2019FoncteursFaiblementPolynomiaux}, who use both \emph{degree}-like (\cf D{\'e}finitions 1.5 and 1.22) and \emph{height}-like (\cf Proposition 2.3) descriptions of their polynomial functors. It is also used by Randal-Williams and Wahl~\cite{Randal-WilliamsWahl2017Homologicalstabilityautomorphism} in their general framework for proving twisted homological stability theorems for sequences of groups.

\paragraph*{Outline.}

In section \ref{sec:inductive-degree} we describe a general framework for defining the \emph{degree} of a functor $\cC \to \cA$ using the structure of an endofunctor $s$ of $\cC$ together with a natural transformation $\mathrm{id} \to s$ (more generally, a collection of such data), and specialise this to several settings in the literature, including symmetric monoidal categories (\S\ref{para:degree-DV}), labelled braid categories (\S\ref{para:partial-braid-categories}) and categories of decorated surfaces (\S\ref{para:degree-mcg}). In section \ref{sec:cross-effects} we set up a general framework for defining the degree of a functor using cross-effects (which we call the \emph{height} of the functor), and specialise this to various settings in the literature, including symmetric monoidal categories (\S\S\ref{para:specialise-DV}--\ref{para:finite-coproducts}), wreath products of categories (\S\ref{para:specialise-CDG}) and labelled braid categories (\S\ref{para:specialise-this-paper}). In sections \ref{sec:functorial-configuration-spaces} and \ref{sec:representations-of-categories} we consider the special case of labelled braid categories $\cC = \cB(M,X)$ in more detail, and describe a functorial (in $M$ and $X$) setting for the notions of degree and height of functors $\cB(M,X) \to \cA$.


\section{Recursive degree}\label{sec:inductive-degree}

To relate different notions of \emph{degree} in the literature, we use a notion of \emph{category with stabilisers}, which is roughly a category $\cC$ equipped with endofunctors $s_i \colon \cC \to \cC$ and natural transformations $\mathrm{id} \to s_i$. These are the objects of a category $\catst$ (see Definition \ref{def:degree-general}). There is then a natural notion of \emph{degree} for any functor $\cC \to \cA$ with $\cC \in \catst$ and $\cA$ an abelian category. There is a functor $\Monini \to \catst$ which is compatible with the definition of \cite{DjamentVespa2019FoncteursFaiblementPolynomiaux} of the degree of a functor with source a monoidal category with initial unit object (\S \ref{para:degree-DV}). This construction also generalises to left modules over such a monoidal category (Remark \ref{rmk:modules-over-monoidal-categories}). There is another functor $\cB \colon \mfdc \to \catst$, which we will define later in \S \ref{sss:some-functors},\footnote{The functor that we define later in fact has source $\mfdc \times \topo$ and target $\cats$, so we are implicitly composing with the inclusions $M \mapsto (M,*) : \mfdc \to \mfdc \times \topo$ and $\cats \subset \catst$, where $* \in \topo$ is the one-point space.} where $\mfdc$ is a category whose objects are smooth manifolds-with-boundary equipped with a collar neighbourhood and a basepoint on the boundary. The operation of boundary connected sum of manifolds gives $\cB(\bD^n)$ the structure of a monoidal category (with initial unit object) and $\cB(M)$ the structure of a left module over it, so $\cB(M)$ may also be viewed as an object of $\catst$ using the previous construction. These two objects of $\catst$ are not equal, but nevertheless result in the same definition of the \emph{degree} of a functor $\cB(M) \to \cA$ (this is Proposition \ref{p:two-degrees-agree}). See \S\S\ref{para:gen-def-degree}--\ref{para:partial-braid-categories} for the details of this brief summary. In \S\S\ref{para:degree-WRW} and \ref{para:degree-mcg} we also discuss how the general Definition \ref{def:degree-general} relates to the notion of degree used in \cite{Randal-WilliamsWahl2017Homologicalstabilityautomorphism} and the notions of degree used in relation to mapping class groups. Throughout this section, $\cA$ will denote a fixed abelian category.

\subsection{A general definition.}\label{para:gen-def-degree}

\begin{defn}\label{def:degree-general}
Let $\catst$ be the category whose objects are small $1$-categories $\cC$ equipped with a collection $\{s_i \colon \cC \to \cC\}_{i \in I}$ of endofunctors and natural transformations $\{\imath_i \colon \mathrm{id} \to s_i\}_{i \in I}$. We call such an object a \emph{category with stabilisers}. A morphism in $\catst$ from $(\cC,I,s,\imath)$ to $(\cD,J,t,\jmath)$ is a functor $f \colon \cC \to \cD$ together with a function $\sigma \colon I \to J$ and a collection of natural isomorphisms $\{ \psi_i \colon f \circ s_i \to t_{\sigma(i)} \circ f \}_{i \in I}$ such that $\jmath_{\sigma(i)} * \mathrm{id}_f = \psi_i \circ (\mathrm{id}_f * \imath_i)$ for all $i \in I$. We denote by $\cats$ the full subcategory on objects where $\lvert I \rvert = 1$ (i.e., categories with just one stabiliser -- we will restrict to this subcategory later, in \S\ref{sec:functorial-configuration-spaces}). It also contains $\mathsf{Cat}$ as the full subcategory on objects where $I = \varnothing$, but this will not be relevant for us.

We define the \emph{degree} of functors from $\cC \in \catst$ to the abelian category $\cA$ as follows. The function $\mathrm{deg} \colon \mathsf{Fun}(\cC,\cA) \to \{-1,0,1,\ldots,\infty\}$ is the largest function such that $\mathrm{deg}(0) = -1$ and such that for non-zero $T$ we have $\mathrm{deg}(T) \leq d$ if and only if
\begin{equation}\label{eq:degree-condition}
\mathrm{deg}(\mathrm{coker}(T\imath_i \colon T \to Ts_i)) \leq d-1,
\end{equation}
for all $i$.

We may also vary the definition slightly and define the \emph{split degree} to be the largest function $\mathrm{sdeg}\colon \mathsf{Fun}(\cC,\cA) \to \{-1,0,1,\ldots,\infty\}$ such that $\mathrm{sdeg}(0) = -1$ and such that for non-zero $T$ we have $\mathrm{sdeg}(T) \leq d$ if and only if
\begin{equation}\label{eq:split-degree-condition}
\begin{gathered}
T\imath_i \colon T \to Ts_i \text{ is a split monomorphism in } \mathsf{Fun}(\cC,\cA) \text{ and } \\ \mathrm{sdeg}(\mathrm{coker}(T\imath_i \colon T \to Ts_i)) \leq d-1,
\end{gathered}
\end{equation}
for all $i$. In between these two definitions, there is the \emph{injective degree} $\mathrm{ideg}(T)$, where the condition that $T\imath_i$ is a split monomorphism in $\mathsf{Fun}(\cC,\cA)$ is weakened to the condition that $\mathrm{ker}(T\imath_i) = 0$.

Another variation of the definition is inspired by the notion of weak degree (\emph{degr{\'e} faible}) introduced by Djament and Vespa~\cite{DjamentVespa2019FoncteursFaiblementPolynomiaux}. Note that $\mathrm{ker}(T\imath_i \colon T \to Ts_i)$ is a subobject of $T$ in the abelian category $\mathsf{Fun}(\cC,\cA)$ for all $i$, and therefore so is the sum $\sum_i \mathrm{ker}(T\imath_i \colon T \to Ts_i)$, which we denote by $\kappa(T)$, following the notation of \cite{DjamentVespa2019FoncteursFaiblementPolynomiaux}. We then define the \emph{weak degree} to be the largest function $\mathrm{wdeg}\colon \mathsf{Fun}(\cC,\cA) \to \{-1,0,1,\ldots,\infty\}$ such that
\[
\mathrm{wdeg}(T) = -1 \qquad\text{if and only if}\qquad \kappa(T) = T,
\]
and otherwise we have $\mathrm{wdeg}(T) \leq d$ if and only if $\mathrm{wdeg}(\mathrm{coker}(T\imath_i \colon T \to Ts_i)) \leq d-1$ for all $i$.
\end{defn}

\begin{rmk}\label{rmk:4-definitions-of-deg}
A simple inductive argument shows that
\[
\mathrm{wdeg}(T) \leq \mathrm{deg}(T) \leq \mathrm{ideg}(T) \leq \mathrm{sdeg}(T)
\]
for all functors $T\colon \cC \to \cA$. Moreover, if $\cC \in \catst$ has the property that each $\imath_i$ has a left-inverse, i.e., a natural transformation $\pi_i \colon s_i \to \mathrm{id}$ such that $\pi_i \circ \imath_i = \mathrm{id}_{\mathrm{id}}$, then all four types of degree are equal for all functors $T \colon \cC \to \cA$.
\end{rmk}

\begin{rmk}\label{rmk:gen-of-degree-under-composition}
In \S\ref{sss:degree} below we discuss the question of when the degree of a functor $\cC \to \cA$ is preserved under precomposition, in the setting where $\cC \in \cats \subset \catst$.\footnote{In \S\ref{sec:functorial-configuration-spaces} we also set $\cA = \ab$, but the only reason for this is to preserve notational similarity with \cite{Palmer2018Twistedhomologicalstability}, and everything in that section generalises verbatim to the setting of an arbitrary abelian category $\cA$.} That discussion extends easily to the setting of $\catst$, and also to the other variations of \emph{degree} defined above, so, for completeness, we mention the general statement here. Let $f = (f,\sigma,\psi) \colon \cC \to \cD$ be a morphism in $\catst$. Lemma \ref{l:degree-under-composition} generalises to say that for any functor $T \colon \cD \to \cA$ we have $\placeholder\mathrm{deg}(Tf) \leq \placeholder\mathrm{deg}(T)$, with equality if $f$ is essentially surjective on objects and $\sigma$ is surjective, for $\placeholder \in \{ \varnothing, \text{i}, \text{s} \}$. For the weak degree we have $\text{wdeg}(Tf) \leq \text{wdeg}(T)$ if $\sigma$ is surjective, and we have equality $\text{wdeg}(Tf) = \text{wdeg}(T)$ if $\sigma$ is bijective and $f$ is essentially surjective on objects. We may then generalise Definition \ref{def:braidable} by saying that an object $(\cC,I,s,\imath)$ of $\catst$ is \emph{braidable} if there are certain natural isomorphisms $\Psi_i \colon s_i \circ s_i \to s_i \circ s_i$ for each $i \in I$. Corollary \ref{coro:braidable} generalises exactly as stated to objects of $\catst$.
\end{rmk}

\begin{rmk}\label{rmk:degree-under-composition-DV}
The above remark, specialised to the setting of Djament and Vespa (see below) and with $\placeholder = \varnothing$, recovers Proposition 1.7 of \cite{DjamentVespa2019FoncteursFaiblementPolynomiaux}. With $\placeholder = \text{w}$, it implies the analogous statement for the weak degree of functors from a monoidal category with initial unit object. In the notation of \cite{DjamentVespa2019FoncteursFaiblementPolynomiaux}, this says that if $\alpha \colon \cM \to \cM^\prime$ is a strict monoidal functor between strict monoidal categories whose unit objects are initial, and if $\alpha$ is moreover surjective on objects, then it induces a functor $\mathcal{P}\mathit{ol}_n(\cM^\prime,\cA) \to \mathcal{P}\mathit{ol}_n(\cM,\cA)$.
\end{rmk}

\subsection{Specialising to the setting of Djament and Vespa.}\label{para:degree-DV}

In the article \cite{DjamentVespa2019FoncteursFaiblementPolynomiaux}, Djament and Vespa work with the category $\Monini$ whose objects are small strict symmetric monoidal categories whose unit object is initial, and whose morphisms are strict monoidal functors. Now, one may define a functor
\[
\Psi \colon \Monini \longrightarrow \catst
\]
as follows: the underlying category of $\Psi(\cM)$ is just $\cM$ and the indexing set for the collection of endofunctors is the set $\mathrm{ob}(\cM)$ of objects of $\cM$. For each $x \in \mathrm{ob}(\cM)$, the endofunctor $s_x \colon \cM \to \cM$ is $x \oplus -$ and the natural transformation $\imath_x \colon \mathrm{id} \to s_x$ consists of the morphisms $i_x \oplus \mathrm{id}_y \colon y = 0 \oplus y \to x \oplus y$, where $i_x \colon 0 \to x$ is the unique morphism from the initial object $0$ to $x$. If $F \colon \cM \to \cN$ is a strict monoidal functor, then $\Psi(F) \colon \Psi(\cM) \to \Psi(\cN)$ is simply the functor $F$, together with the function $\mathrm{ob}(F)$ from the indexing set $\mathrm{ob}(\cM)$ of $\Psi(\cM)$ to the indexing set $\mathrm{ob}(\cN)$ of $\Psi(\cN)$, and the natural isomorphisms are identities.

Given $\cM \in \mathrm{ob}(\Monini)$, an abelian category $\cA$ and a functor $T \colon \cM \to \cA$, we may view $\cM$ as an object of $\catst$ via the functor $\Psi$, and therefore obtain notions of \emph{degree} $\mathrm{deg}(T)$ and \emph{weak degree} $\mathrm{wdeg}(T)$. These coincide with the definitions of \emph{strong degree} and \emph{weak degree}, introduced in \cite{DjamentVespa2019FoncteursFaiblementPolynomiaux}, respectively (\cf D{\'e}finition 1.5 for the strong degree, and for the weak degree see D{\'e}finitions 1.10, 1.16 and 1.22, as well as Proposition 1.19, which provides the key property -- using the notation of \cite{DjamentVespa2019FoncteursFaiblementPolynomiaux} -- that $\delta_x$ and $\pi_\cM$ commute).

The article \cite{DjamentVespa2019FoncteursFaiblementPolynomiaux} in fact sets up a detailed theory of \emph{weak polynomial functors} (those with finite weak degree) by considering the quotient category $\mathsf{Fun}(\cM,\cA)/\Sn(\cM,\cA)$, where $\Sn(\cM,\cA)$ is the full subcategory of functors $T$ with $\mathrm{wdeg}(T)<0$. Since the notion of weak degree may be described very generally, whenever the source category is an object of $\catst$, it may be interesting to try to export this theory from $\Monini$ to other settings to which the general definition for $\catst$ specialises, such as twisted coefficient systems for mapping class groups (\cf \S\ref{para:degree-mcg} below).

\begin{rmk}
We note that the construction $\Psi$ above does not use the symmetry of $\cM \in \Monini$, and in fact works equally well for any strict monoidal category whose unit object is initial. Another remark is that, if the unit object of $\cM$ is null, i.e., initial and terminal, then the natural transformations $\imath_x \colon \mathrm{id} \to s_x$ have left-inverses $\pi_x \colon s_x \to \mathrm{id}$ given by the morphisms $t_x \oplus \mathrm{id}_y \colon x \oplus y \to 0 \oplus y = y$, where $t_x \colon x \to 0$ is the unique morphism from $x$ to the terminal object $0$. So for functors $\cM \to \cA$ from a monoidal category with null unit object, the three types of degree coincide, by Remark \ref{rmk:4-definitions-of-deg}.
\end{rmk}

\begin{rmk}[\textit{Modules over monoidal categories.}]\label{rmk:modules-over-monoidal-categories}
Recall that a strict left-module over a strict monoidal category $\cM$ is a category $\cC$ and a functor ${\oplus} \colon \cM \times \cC \to \cC$ such that ${\oplus} \circ (\mathbf{1}_{\cM} \times \mathrm{id}_{\cC}) = \mathrm{id}_{\cC}$ and ${\oplus} \circ (\mathrm{id}_{\cM} \times {\oplus}) = {\oplus} \circ ({\oplus} \times \mathrm{id}_{\cC})$, where $\mathbf{1}_{\cM} \colon * \to \cM$ takes the unique object to the unit object $I_{\cM}$ of $\cM$. If $I_{\cM}$ is initial in $\cM$, then any strict left-module $\cC$ over $\cM$ naturally has the structure of an object of $\catst$, generalising exactly the construction above, which is the case of $\cM$ as a module over itself: the indexing set is $\mathrm{ob}(\cM)$, the endomorphisms are defined by $x \oplus -$ and the natural transformations are formed using the fact that $I_{\cM}$ is initial.\footnote{For the author, the idea of generalising from monoidal categories to modules over monoidal categories came from a conversation in 2015 with Aur{\'e}lien Djament.}
\end{rmk}

\subsection{Partial braid categories.}\label{para:partial-braid-categories}

In \S\ref{sec:functorial-configuration-spaces} below we define another functor
\[
\cB \colon \mfdc \times \topo \longrightarrow \cats \subset \catst
\]
sending a manifold $M$ (equipped with a collar neighbourhood and a basepoint on its boundary) and a space $X$ to the (labelled) \emph{partial braid category} $\cB(M,X)$, whose objects are the non-negative integers. See \S\S\ref{sss:some-categories} and \ref{sss:some-functors} for the full details of this construction (alternatively \S\ref{para:degree-WRW} for a description of the underlying category $\cB(M,X)$, without the functoriality or the structure as an object of $\cats$). For the next few paragraphs we will denote this object instead by $\cB(M,X)^{\firsttype}$ in order to distinguish it from a different structure (which we will define next) on the same underlying category, also making it into an object of $\catst$.

For $n\geq 2$ let $\bD^n$ denote the closed unit disc in Euclidean $n$-space, equipped with a collar neighbourhood and basepoint on its boundary. This is an object of $\mfdc$, and for any $X \in \topo$ the category $\cB(\bD^n,X)$ can be made into a strict monoidal category with the number zero as its (null) unit object. For any object $M$ of $\mfdc$ of dimension $n$, the category $\cB(M,X)$ then has the structure of a strict left-module over $\cB(\bD^n,X)$. Both the monoidal and the module structure are induced by the operation of boundary connected sum of two manifolds in $\mfdc$. Thus, by Remark \ref{rmk:modules-over-monoidal-categories} above, there is another structure on $\cB(M,X)$ making it into an object of $\catst$, coming from this module structure. Denote this object of $\catst$ by $\cB(M,X)^{\secondtype}$.

The objects $\cB(M,X)^{\firsttype}$ and $\cB(M,X)^{\secondtype}$ of $\catst$ have the same underlying category $\cB(M,X)$, so any functor $T \colon \cB(M,X) \to \cA$ has a degree with respect to each of these structures; denote these by $\mathrm{deg}^{\firsttype}(T)$ and $\mathrm{deg}^{\secondtype}(T)$ respectively.

\begin{prop}\label{p:two-degrees-agree}
In this setting, for any functor $T \colon \cB(M,X) \to \cA$, we have $\mathrm{deg}^{\firsttype}(T) = \mathrm{deg}^{\secondtype}(T)$.
\end{prop}

We will prove this as a corollary of a more general statement about modules over monoidal categories. For any object $(\cC,I,s,\imath) \in \catst$ and functor $T \colon \cC \to \cA$, we have a degree $\mathrm{deg}(T)$. But for any element $x \in I$ we may also forget part of the structure, considering just the object $(\cC,s_x,\imath_x) \in \cats$, and compute the degree of $T$ with respect to this structure -- denote this by $\mathrm{deg}^x(T)$. An easy inductive argument shows that $\mathrm{deg}^x(T) \leq \mathrm{deg}(T)$.

\begin{prop}\label{p:two-degrees-agree-2}
Let $\cC$ be a strict left-module over a strict braided monoidal category $\cM$, whose unit object $I_{\cM}$ is null, and which is generated by $x \in \mathrm{ob}(\cM)$, in the sense that every object of $\cM$ is isomorphic to $x^{\oplus n}$ for some non-negative integer $n$. Consider $\cC$ as an object of $\catst$ as in Remark \ref{rmk:modules-over-monoidal-categories} and let $T \colon \cC \to \cA$ be a functor. Then $\mathrm{deg}^x(T) = \mathrm{deg}(T)$.
\end{prop}

We prove an analogous comparison result for heights in Proposition \ref{prop:compare-two-heights-special-case}. See Remark \ref{rmk:summary} for a summary of how these facts are related. Also see Remark \ref{rmk:pre-braided} for generalisations of Proposition \ref{p:two-degrees-agree-2} and references to related results.

\begin{proof}[Proof of Proposition \ref{p:two-degrees-agree}]
First note that the monoidal category $\cB(\bD^n,X)$ is braided (since $n\geq 2$) and is generated by the object $1$. Thus the category $\cC = \cB(M,X)^{\secondtype}$ satisfies the hypotheses of Proposition \ref{p:two-degrees-agree-2}, which implies that $\mathrm{deg}^{\secondtype}(T) = \mathrm{deg}(T) = \mathrm{deg}^{1}(T) = \mathrm{deg}^{\firsttype}(T)$.\footnote{For the final equality $\mathrm{deg}^{1}(T) = \mathrm{deg}^{\firsttype}(T)$ to be valid, one has to be slightly more precise with the definition of the structure of $\cB(M,X)$ as a module over $\cB(\bD^n,X)$: it must be induced by the boundary connected sum between $\bD^n$ and $M$, \emph{using the component of $\partial M$ containing the basepoint}.}
\end{proof}

To prove Proposition \ref{p:two-degrees-agree-2}, we first establish a lemma, which will allow us to apply Corollary \ref{coro:braidable} from \S\ref{sec:functorial-configuration-spaces} below in the present setting. Let $\cC$ be as in the statement of the proposition, considered as an object of $\catst$, i.e., equipped with an endofunctor and natural transformation $\iota_y \colon \mathrm{id}_{\cC} \Rightarrow y \oplus -$ for each object $y$ of $\cM$. Write $\cC^x$ for the object $(\cC,x \oplus -,\iota_x)$ of $\cats$, where we have forgotten all but one of the endofunctors. (For example if $\cC = \cB(M,X)^{\secondtype}$ and $x = 1$ then $\cC^x = \cB(M,X)^{\firsttype}$.)

\begin{lem}\label{lem:braidable}
The object $\cC^x \in \cats$ is braidable in the sense of Definition \ref{def:braidable}.
\end{lem}

\begin{proof}
We need to find a certain natural automorphism $\Psi$ of the functor $s \circ s = x \oplus x \oplus -$. Note that $\imath * \mathrm{id}_s$ and $\mathrm{id}_s * \imath$ are the natural transformations $s \Rightarrow s \circ s$ consisting of the morphisms $x \oplus c \to x \oplus x \oplus c$, for each $c \in \mathrm{ob}(\cC)$, given by the matrices
\begin{equation}\label{eq:two-matrices}
\left(\,
\begin{matrix}
0 & 0 \\
\mathrm{id}_x & 0 \\
0 & \mathrm{id}_c
\end{matrix}
\,\right)
\qquad\text{and}\qquad
\left(\,
\begin{matrix}
\mathrm{id}_x & 0 \\
0 & 0 \\
0 & \mathrm{id}_c
\end{matrix}
\,\right)
\end{equation}
respectively. We need to show that these differ by a natural automorphism $\Psi$.
This may be constructed from the braiding of $\cM$, as follows. Write $i$ for the inclusion $\cC \to \cM \times \cM \times \cC$ given by $c \mapsto (x,x,c)$ and write $f$ for the flip functor $\cM \times \cM \to \cM \times \cM$ given by $(y,z) \mapsto (z,y)$. Then the braiding of $\cM$ is a natural isomorphism $b \colon {\oplus} \Rightarrow {\oplus} \circ f \colon \cM \times \cM \to \cM$. Taking products with $\cC$ and identities, this induces a natural isomorphism $b \times \mathrm{id} \colon {\oplus} \times \mathrm{id}_{\cC} \Rightarrow ({\oplus} \circ f) \times \mathrm{id}_{\cC} \colon \cM \times \cM \times \cC \to \cM \times \cC$. Then we may take $\Psi$ to be the automorphism $\oplus * (b \times \mathrm{id}) * i$ of $x \oplus x \oplus -$. Diagrammatically:
\begin{equation}\label{eq:natural-aut}
\centering
\begin{split}
\begin{tikzpicture}
[x=1mm,y=1mm]
\node (l) at (-5,0) {$\cC$};
\node (ml) at (20,0) {$\cM \times \cM \times \cC$};
\node (m) at (60,-10) {$\cM \times \cM \times \cC$};
\node (mr) at (100,0) {$\cM \times \cC$};
\node (r) at (120,0) {$\cC.$};
\draw[->] (l) to node[above,font=\small]{$i$} (ml);
\draw[->] (ml.east) to node[above,font=\small]{$\oplus \times \mathrm{id}_\cC$} (mr.west);
\draw[->] (ml.south east) to[out=-45,in=180] (m.west);
\draw[->] (m.east) to[out=0,in=225] (mr.south west);
\node at (88,-8) [anchor=west] {$\oplus \times \mathrm{id}_\cC$};
\node at (36,-8) [anchor=east] {$f \times \mathrm{id}_\cC$};
\draw[->] (mr) to node[above,font=\small]{$\oplus$} (r);
\draw[double,double equal sign distance,-implies] (60,-2) to node[right,font=\small]{$b \times \mathrm{id}$} (m);
\end{tikzpicture}
\end{split}
\end{equation}
In components, we may write this as the collection of morphisms $b_{x,x} \oplus \mathrm{id}_c$ for $c \in \mathrm{ob}(\cC)$, where $b_{x,x}$ denotes the braiding of $\cM$ on the object $x$. The fact that $\imath * \mathrm{id}_s$ and $\mathrm{id}_s * \imath$ differ by $\Psi$ then follows from the equation:
\begin{equation}\label{eq:natural-aut-equation}
\left(\,
\begin{array}{ccc}
\multicolumn{2}{c}{\multirow{2}{*}{$b_{x,x}$}} & 0 \\
&& 0 \\
0 & 0 & \mathrm{id}_c
\end{array}
\,\right)
\cdot
\left(\,
\begin{matrix}
\mathrm{id}_x & 0 \\
0 & 0 \\
0 & \mathrm{id}_c
\end{matrix}
\,\right)
\quad = \quad
\left(\,
\begin{matrix}
0 & 0 \\
\mathrm{id}_x & 0 \\
0 & \mathrm{id}_c
\end{matrix}
\,\right) ,
\end{equation}
where we are using the matrix notation of \eqref{eq:two-matrices}.
\end{proof}

\begin{proof}[Proof of Proposition \ref{p:two-degrees-agree-2}]
It is always true that $\mathrm{deg}^x(T) \leq \mathrm{deg}(T)$, as observed just before the statement of Proposition \ref{p:two-degrees-agree-2}. So we just need to prove, for all $d\geq -1$, that, if $\mathrm{deg}^x(T) \leq d$, then $\mathrm{deg}(T) \leq d$. The proof will be by induction on $d$. The base case, when $d=-1$, is clear, since both statements are equivalent to $T$ being equal to the zero functor.

Now let $d\geq 0$ and assume by induction that the implication is true for smaller values of $d$. We assume that $\mathrm{deg}^x(T) \leq d$ and we need to show that $\mathrm{deg}(T) \leq d$. We showed in Lemma~\ref{lem:braidable} that $\cC^x = (\cC,s_x,\imath_x) \in \cats$ is braidable, so Corollary \ref{coro:braidable} implies that $\mathrm{deg}^x(T \circ s_x) \leq \mathrm{deg}^x(T)$. Here we are writing $s_x$ as shorthand for $x \oplus -$. Iterating this argument, we see that
\[
\mathrm{deg}^x(T \circ (s_x)^i) \leq \mathrm{deg}^x(T) \leq d
\]
for all $i\geq 0$. By the recursive definition of $\mathrm{deg}^x(-)$, this means that
\[
\mathrm{deg}^x \bigl( \mathrm{coker} \bigl( T \circ (s_x)^i * \imath_x \colon T \circ (s_x)^i \longrightarrow T \circ (s_x)^{i+1} \bigr) \bigr) \leq d-1
\]
for all $i\geq 0$. Since the unit object $I_{\cM}$ of $\cM$ is null, not just initial, we know that the natural transformations $\imath_y$, for objects $y \in \mathrm{ob}(\cM)$, are all split-injective. Now, for any $n\geq 0$, the natural transformation $T * \imath_{x^{\oplus n}}$ is equal to the composition
\[
\bigl( T \circ (s_x)^{n-1} * \imath_x \bigr) \circ \quad\cdots\cdots\quad \circ \bigl( T \circ (s_x)^2 * \imath_x \bigr) \circ \bigl( T \circ s_x * \imath_x \bigr) \circ \bigl( T * \imath_x \bigr) .
\]
This is a composition of split-injective morphisms in the abelian category $\mathsf{Fun}(\cC,\cA)$, so we have
\begin{equation}\label{eq:decomposition-of-cokernels}
\mathrm{coker}(T * \imath_{x^{\oplus n}}) \;\cong\; \bigoplus_{i=0}^{n-1} \; \mathrm{coker}(T \circ (s_x)^i * \imath_x)
\end{equation}
and hence
\[
\mathrm{deg}^x(\mathrm{coker}(T * \imath_{x^{\oplus n}})) \;=\; \max_{i=0,\ldots,n-1} \bigl( \mathrm{deg}^x (\mathrm{coker}(T \circ (s_x)^i * \imath_x)) \bigr) \leq d-1.
\]
Now let $y$ be any object of $\cM$. By assumption, $y$ is isomorphic to $x^{\oplus n}$ for some $n\geq 0$. Thus there is a natural isomorphism $\Phi \colon T \circ s_{x^{\oplus n}} \to T \circ s_y$ such that $\Phi \circ (T * \imath_{x^{\oplus n}}) = T * \imath_y$, and so
\[
\mathrm{coker}(T * \imath_y) \;\cong\; \mathrm{coker}(T * \imath_{x^{\oplus n}}).
\]
Thus, for any object $y$ of $\cM$, we have $\mathrm{deg}^x(\mathrm{coker}(T * \imath_y)) \leq d-1$. By the inductive hypothesis we therefore also have, for any $y \in \mathrm{ob}(\cM)$,
\[
\mathrm{deg}(\mathrm{coker}(T * \imath_y)) \leq d-1.
\]
By the recursive definition of $\mathrm{deg}(-)$, this implies that $\mathrm{deg}(T) \leq d$.
\end{proof}

\begin{rmk}[\textit{Generalisations.}]\label{rmk:pre-braided}\footnote{The author would like to thank Aur{\'e}lien Djament for pointing out an error in an earlier version of this remark.}
For Lemma \ref{lem:braidable}, and thus for Proposition \ref{p:two-degrees-agree-2}, it is possible to weaken the assumption that $\cM$ is braided to the assumption that it is \emph{pre-braided} (a notion that was introduced in \cite[Definition 1.5]{Randal-WilliamsWahl2017Homologicalstabilityautomorphism}). By definition, this means that its underlying groupoid $\cM^{\sim}$ is braided and the braiding $b_{x,y} \colon x \oplus y \to y \oplus x$ of $\cM^{\sim}$ satisfies the equation
\begin{equation}\label{eq:pre-braided}
\left(\,
\begin{array}{cc}
\multicolumn{2}{c}{\multirow{2}{*}{$b_{x,y}$}} \\
&
\end{array}
\,\right)
\cdot
\left(\,
\begin{matrix}
\mathrm{id}_x \\
0
\end{matrix}
\,\right)
\quad = \quad
\left(\,
\begin{matrix}
0 \\
\mathrm{id}_x
\end{matrix}
\,\right) \; \colon \; x \longrightarrow y \oplus x,
\end{equation}
for any two objects $x,y$ of $\cM$. The existence of the braiding on $\cM^{\sim}$ allows one to construct the automorphism $\Psi$ (replace each appearance of $\cM$ with $\cM^{\sim}$ in the diagram \eqref{eq:natural-aut}) and the relation \eqref{eq:pre-braided} implies the relation \eqref{eq:natural-aut-equation}. By the same reasoning, we could dually weaken the assumption that $\cM$ is braided to the assumption that it is \emph{pre\textsuperscript{op}-braided}, meaning that $\cM^{\sim}$ is braided and its braiding satisfies the equation
\begin{equation}\label{eq:preop-braided}
\left(\,
\begin{array}{cc}
\multicolumn{2}{c}{\multirow{2}{*}{$b_{x,y}$}} \\
&
\end{array}
\,\right)
\cdot
\left(\,
\begin{matrix}
0 \\
\mathrm{id}_y
\end{matrix}
\,\right)
\quad = \quad
\left(\,
\begin{matrix}
\mathrm{id}_y \\
0
\end{matrix}
\,\right) \; \colon \; y \longrightarrow y \oplus x,
\end{equation}
for any two objects $x,y$ of $\cM$.

The assumption that $I_{\cM}$ is null in Proposition \ref{p:two-degrees-agree-2} was convenient to make the homological algebra simpler, by giving us the decomposition \eqref{eq:decomposition-of-cokernels}, but we expect the proposition to hold more generally whenever $I_{\cM}$ is initial (\cf Proposition 1.8 of \cite{DjamentVespa2019FoncteursFaiblementPolynomiaux}; see also Proposition 3.9 of \cite{Soulie2017LongMoodyconstruction}). One can of course also generalise this proposition to the setting in which $\cM$ has a given \emph{set} of objects that generate it, instead of a single object (\cf the two references just cited).

When $I_{\cM}$ is not null, the four versions of degree defined in \S\ref{para:gen-def-degree} do not necessarily coincide, so one may ask whether Proposition \ref{p:two-degrees-agree-2} is also true if deg is replaced by $\placeholder\text{deg}$ for $\placeholder \in \{ \text{i} , \text{s} , \text{w} \}$. For the weak degree ($\placeholder = \text{w}$) this is true, by Proposition 1.24 of \cite{DjamentVespa2019FoncteursFaiblementPolynomiaux} (their statement is for a symmetric monoidal category, rather than a left-module over a pre-braided monoidal category, but their methods should extend to this more general setting too), and for $\placeholder = \text{i or s}$ the above proof goes through with minor modifications, making use of Remark \ref{rmk:gen-of-degree-under-composition}.
\end{rmk}

\subsection{Relation to the degree of Randal-Williams and Wahl.}\label{para:degree-WRW}

In their paper \cite{Randal-WilliamsWahl2017Homologicalstabilityautomorphism}, Randal-Williams and Wahl use a notion of degree of twisted coefficient systems which is slightly different to that of \cite{DjamentVespa2019FoncteursFaiblementPolynomiaux}, and which they remark is inspired by the work of Dwyer~\cite{Dwyer1980Twistedhomologicalstability}, van der Kallen~\cite{Kallen1980Homologystabilitylinear} and Ivanov~\cite{Ivanov1993homologystabilityTeichmuller}.

\paragraph*{Setting.}

The starting point in \cite{Randal-WilliamsWahl2017Homologicalstabilityautomorphism} is a \emph{homogeneous category} $\cC$ -- meaning a monoidal category whose unit object is initial, satisfying two axioms H1 and H2 described in Definition 1.3 of \cite{Randal-WilliamsWahl2017Homologicalstabilityautomorphism} -- which is also \emph{pre-braided} -- see Remark \ref{rmk:pre-braided} above -- together with two objects $a$ and $x$ of $\cC$. Let $\cC_{a,x}$ denote the full subcategory on the objects $x^{\oplus m} \oplus a \oplus x^{\oplus n}$. There is an endofunctor of this category given by $x \oplus -$ and a natural transformation $\mathrm{id} \to (x \oplus -)$ since the unit of $\cC$ is initial, so $\cC_{a,x}$ is in this way an object of $\cats \subset \catst$. A twisted coefficient system in \cite{Randal-WilliamsWahl2017Homologicalstabilityautomorphism} is a functor $T \colon \cC_{a,x} \to \cA$. For each $N \geq 0$ they define a notion of \emph{degree at $N$} and \emph{split degree at $N$} for $T$ (see Definition 4.10); when $N=0$ these correspond to the injective degree and the split degree of $T$ as defined in \S\ref{para:gen-def-degree}.

\begin{rmk}[\textit{Comparison to degree for modules over monoidal categories.\footnote{The author is grateful to Manuel Krannich for the observation that $\cC_{a,x}$ can be viewed as a module over a monoidal category.}}]
If we denote by $\cC_x$ the full (monoidal) subcategory of $\cC$ on the objects $x^{\oplus n}$ for $n\geq 0$, then $\cC_{a,x}$ is a left-module over $\cC_x$, so there is a notion of (injective, split, etc.) degree of functors $\cC_{a,x} \to \cA$ coming from Remark \ref{rmk:modules-over-monoidal-categories}. If the unit object of $\cC$ is null,\footnote{We expect that this assumption is not necessary, since we expect that Proposition \ref{p:two-degrees-agree-2} should hold assuming only that $I_{\cM}$ is initial, and also with deg replaced by either ideg or sdeg (the case of wdeg seems more subtle).} this exactly coincides with the degree (at $N=0$) of \cite{Randal-WilliamsWahl2017Homologicalstabilityautomorphism}. To see this, note first that the degree of $T$ (at $N=0$) according to \cite{Randal-WilliamsWahl2017Homologicalstabilityautomorphism} is precisely $\mathrm{deg}^x(T)$ in the notation of Proposition \ref{p:two-degrees-agree-2}.\footnote{The four types of degree coincide since the unit object of $\cC$ is null (\cf Remark \ref{rmk:4-definitions-of-deg}), which is why we can write $\mathrm{deg}^x(T)$ instead of, say, $\mathrm{ideg}^x(T)$.} Then Proposition \ref{p:two-degrees-agree-2} plus Remark \ref{rmk:pre-braided} imply that this is equal to the degree of $T$ according to the structure of $\cC_{a,x}$ as a module over $\cC_x$.
\end{rmk}

\paragraph*{Injective braid categories.}

As mentioned above (\S\ref{para:partial-braid-categories}), we define in \S\ref{sec:functorial-configuration-spaces} below the \emph{partial braid category} $\cB(M,X)$ associated to a manifold $M$ and space $X$, which is naturally a category with stabiliser, in other words, an object of $\cats$. It may be described as follows: its objects are finite subsets $c$ of $M$ equipped with a function $\ell \colon c \to X$ (``labelled by $X$''). A morphism from $\ell \colon c \to X$ to $m \colon d \to X$ is a \emph{braid between subconfigurations of $c$ and $d$ labelled by paths in $X$}. More precisely, it is a path $\gamma$ in the configuration space $C_k(M,X)$ for some integer $k$, up to endpoint-preserving homotopy, such that $\gamma(0)$ is the restriction of $\ell$ to some subset of $c$ and $\gamma(1)$ is the restriction of $m$ to some subset of $d$.

In fact, this defines a slightly larger category $\hat{\cB}(M,X)$, of which $\cB(M,X)$ is a skeleton. Both $M$ and $X$ are assumed to be path-connected, so the isomorphism classes of the objects $\ell \colon c \to X$ of $\hat{\cB}(M,X)$ are determined by the cardinality $\lvert c \rvert$. Then $\cB(M,X)$ is the full subcategory on the objects $c_n \to \{x_0\} \subseteq X$, where $x_0$ is the basepoint of $X$ and $c_n$ is a certain nested sequence of subsets of $M$ of cardinality $n$. We may therefore think of the objects of $\cB(M,X)$ as the non-negative integers. See \S\S\ref{sss:some-categories} and \ref{sss:some-functors} for more details, including the functoriality of this definition with respect to $M$ and $X$ and the structure making $\cB(M,X)$ into an object of $\cats$.

There is a subcategory of $\cB(M,X)$, denoted $\cBf(M,X)$ and called the \emph{injective braid category}, also with the non-negative integers as objects, but with only those morphisms (using the description of the previous paragraph) where $\gamma(0)=\ell$. Morphisms in $\cBf(M,X)$ may be thought of as ``fully-defined injective braids on $M$'', whereas those in $\cB(M,X)$ are ``partially-defined injective braids on $M$''. The stabiliser (endofunctor plus natural transformation) of $\cB(M,X)$ restricts to $\cBf(M,X)$, making it into a subobject in the category $\cats$.

The simplest example corresponds to taking $X$ a point and $M=\bR^n$ for $n\geq 3$, in which case $\cBf(M,X)$ is equivalent to the category FI of finite sets and injections, and $\cB(M,X)$ is equivalent to the category $\text{FI}\sharp$ of finite sets and partially-defined injections.

\paragraph*{Which braid categories are homogeneous?}

One may wonder whether the categories $\cB(M,X)$ and $\cBf(M,X)$ are pre-braided homogeneous. First of all, if $M$ splits as $M = \bR \times M^\prime$, they are both monoidal with initial unit object, and if moreover $M^\prime$ also splits as $M^\prime = \bR \times M^{\prime\prime}$ they are braided (and hence pre-braided). The category $\cB(M,X)$ is, however, never homogeneous: it fails axiom H1 for homegeneity. On the other hand, the category $\cBf(M,X)$ always satisfies axiom H1, and it satisfies axiom H2 if and only if $M = \bR^2 \times M^\prime$ has dimension at least $3$, i.e., $\mathrm{dim}(M^\prime) \geq 1$.

In particular, the category $\cBf(\bR^2)$ is not homogeneous. The ``natural'' pre-braided homogeneous category whose automorphisms groups are the braid groups is denoted $U\beta$ in \cite{Randal-WilliamsWahl2017Homologicalstabilityautomorphism}, and comes with a natural functor $U\beta \to \cBf(\bR^2)$. Using the graphical calculus for $U\beta$ described in \S 1.2 of \cite{Randal-WilliamsWahl2017Homologicalstabilityautomorphism}, this functor may be described as taking a braid diagram representing a morphism of $U\beta$ and forgetting all strands with ``free'' ends.

\paragraph*{Comparison of twisted homological stability results.}

As an aside, we discuss briefly the overlap between the twisted homological stability results of \cite{Randal-WilliamsWahl2017Homologicalstabilityautomorphism} and those of \cite{Palmer2018Twistedhomologicalstability} (where this note originated). For the purposes of this paragraph, a sequence of (based, path-connected) spaces $X_n$ indexed by non-negative integers is \emph{homologically stable} if for each $i$ the group $H_i(X_n)$ is independent of $n$ (up to isomorphism) once $n$ is sufficiently large. (Given a sequence of groups $G_n$ we consider their classifying spaces $X_n = BG_n$.) If $\cC$ is a category whose objects are non-negative integers and $\mathrm{Aut}_\cC(n) = \pi_1(X_n)$, then a functor $T \colon \cC \to \ab$ determines a local coefficient system on each $X_n$, and the sequence is homologically stable \emph{with coefficients in $T$} if the corresponding local homology groups stabilise.

Theorem A of \cite{Randal-WilliamsWahl2017Homologicalstabilityautomorphism} says that the groups $\mathrm{Aut}_{\cC}(a \oplus x^{\oplus n})$ are homologically stable with coefficients in any finite-degree twisted coefficient system on $\cC_{a,x}$, as long as $\cC$ is pre-braided homogeneous and a certain simplicial complex built out of $\cC_{a,x}$ is highly-connected.

Taking $\cC = \cBf(M,X)$ and $M = \bR^2 \times M^\prime$ for a manifold $M^\prime$ of dimension at least one, we saw above that $\cC$ is pre-braided homogeneous. Taking $a = 0$ and $x = 1$, we have $\cC_{a,x} = \cC$, which is equivalent to the category $\mathrm{FI}_G$ of \cite{SamSnowden2014RepresentationscategoriesG} with $G=\pi_1(M\times X)$. As noted in \cite{Randal-WilliamsWahl2017Homologicalstabilityautomorphism} (at the bottom of page 596), the associated simplicial complex is known to be highly-connected by a result of \cite{HatcherWahl2010Stabilizationmappingclass}, and so Theorem A of \cite{Randal-WilliamsWahl2017Homologicalstabilityautomorphism} applies in this setting. In fact, it yields a particular case of their Theorem D, saying that the sequence of fundamental groups $G_n = \pi_1(C_n(M,X)) \cong \pi_1(M\times X) \wr \Sigma_n$ satisfies twisted homological stability for finite-degree coefficient systems on the category $\cBf(M,X) = U(\sqcup_n G_n)$. On the other hand, in this setting, Theorem A of \cite{Palmer2018Twistedhomologicalstability} says that the sequence of (\emph{non-aspherical}) spaces $C_n(M,X)$ satisfies twisted homological stability for finite-degree coefficient systems on the larger category $\cB(M,X)$.

If $M=S$ is a surface and $X=BG$ is an aspherical space, then the configuration spaces $C_n(S,BG)$ are also aspherical with fundamental groups $G_n = G\wr \beta_n^S$, where $\beta_n^S$ denotes the $n$th surface braid group. In this case Theorem A of \cite{Palmer2018Twistedhomologicalstability} says that this sequence of groups satisfies twisted homological stability for finite-degree coefficient systems on $\cB(S,BG)$. In this setting, Theorem D of \cite{Randal-WilliamsWahl2017Homologicalstabilityautomorphism} also says that this sequence of groups satisfies twisted homological stability, but for finite-degree coefficient systems on the category $U(\sqcup_n G_n)$. This is more general, since there is a natural functor $U(\sqcup_n G_n) \to \cBf(S,BG) \subset \cB(S,BG)$, and precomposition by this functor preserves the degree of twisted coefficient systems (\cf Lemma \ref{l:degree-under-composition}).

\begin{rmk}
When $M$ has dimension greater than $2$ or when $X$ has non-trivial higher homotopy groups, the spaces $C_n(M,X)$ are not aspherical, so in this setting the twisted homological stability result of \cite{Palmer2018Twistedhomologicalstability} is not comparable to the results of \cite{Randal-WilliamsWahl2017Homologicalstabilityautomorphism}, since the latter paper is concerned only with sequences of \emph{groups}. On the other hand, the framework of \cite{Randal-WilliamsWahl2017Homologicalstabilityautomorphism} has been generalised by Krannich~\cite{Krannich2017Homologicalstabilitytopological} to a topological setting, which includes the setting of configuration spaces, even when they are not aspherical. See Remark 1.5 of \cite{Palmer2018Twistedhomologicalstability} for a comparison.
\end{rmk}

\subsection{Degree of twisted coefficient systems on mapping class groups.}\label{para:degree-mcg}

There are several different settings that have been considered for twisted coefficient systems on mapping class groups and their degrees, all using the notion of ``split degree'' (or slight variations thereof) described in \S\ref{para:gen-def-degree}. We will describe and compare these different settings, using the language of \S\ref{para:gen-def-degree}, without defining in all details the categories involved.

There is a certain category $\cC$, introduced by Ivanov~\cite{Ivanov1993homologystabilityTeichmuller}, whose objects are compact, connected, oriented surfaces $F$ equipped with an embedded arc in $\partial F$. Morphisms are, roughly, embeddings together with a path between the midpoints of the two arcs, all considered up to ambient isotopy. There is an endofunctor $t\colon \cC \to \cC$ and a natural transformation $\mathrm{id} \to t$ defined by Ivanov, which on objects takes the boundary connected sum with $F_{1,1}$, the torus with one boundary component. There is another such endofunctor $a\colon \cC \to \cC$, introduced by Cohen and Madsen~\cite{CohenMadsen2009Surfacesinbackground}, which instead takes the boundary connected sum with an annulus.

The coefficient systems of Ivanov are indexed on $\cC$ and his degree is the \emph{split degree} (as defined in \S\ref{para:gen-def-degree}), considering $\cC$ as an object of $\catst$ using just the endomorphism $t$. Cohen and Madsen use a slight variation of $\cC$ to index their coefficient systems, and their degree is again the split degree, but this time using both $t$ and $a$ to turn $\cC$ into an object of $\catst$. As a side note, their definition very slightly deviates from this, in fact. They do not require that the splittings of $T \to Tt$ and of $T \to Ta$ are functorial, i.e., they do not have to be natural transformations. They only require that $T(F) \to Tt(F)$ and $T(F) \to Ta(F)$ split for each $F$, and that these splittings are equivariant for the action of the automorphism group of the object $F$ in $\cC$ (which is the mapping class group of $F$). In other words, $T \to Tt$ and $T \to Ta$ are only required to be split mono natural transformations after restricting $\cC$ to the subcategory $\cC_{\mathrm{aut}} \subset \cC$ of all automorphisms in $\cC$.

Boldsen~\cite{Boldsen2012Improvedhomologicalstability} uses the same $\cC$ as Cohen and Madsen and the same two endofunctors, and he also introduces another functor $p \colon \cC(2) \to \cC$, defined on a certain subcategory $\cC(2)$ of $\cC$ in which objects all have at least two boundary components, which glues a pair of pants onto two boundary components of a given surface. His coefficient systems are indexed on $\cC$, as for Cohen and Madsen. The endofunctors $t$ and $a$ turn $\cC$ into an object of $\catst$, and therefore give a notion of split degree. However, Boldsen's definition of degree is slightly stricter: the recursive condition \eqref{eq:split-degree-condition} is modified to say that $T \to Tt$ and $T \to Ta$ must be split mono and $\mathrm{sdeg}(\mathrm{coker}(T\to Tt)) \leq d-1$ and $\mathrm{sdeg}(\mathrm{coker}(T\to Ta)) \leq d-1$, and also $T|_{\cC(2)} \to Tp$ must also be split mono, in $\mathsf{Fun}(\cC(2),\cA)$.

Randal-Williams and Wahl also consider mapping class groups as an example of their general twisted homological stability machine, and their setup is again slightly different to the previous settings. They consider two subcategories of $\cC$ separately. One is the full subcategory on surfaces with any genus but a fixed number of boundary components, to which the endofunctor $t$ restricts. They then consider coefficient systems indexed on this subcategory, and define the split degree of such coefficient systems by using the restriction of $t$ to view the subcategory as an object of $\catst$. (For simplicity we are taking $N=0$ in their definition of split degree.) Separately, they consider the subcategory on surfaces with a fixed genus and any number of boundary components, to which the endofunctor $a$ restricts. They then consider coefficient systems indexed on this subcategory, and define the split degree by using the restriction of $a$ to view it as an object of $\catst$. Finally, they also consider a non-orientable analogue of Ivanov's category $\cC$, with objects all non-orientable surfaces with a given fixed number of boundary components and any (non-orientable) genus. This admits an endofunctor $m$ defined on objects by taking the boundary connected sum with a M{\"o}bius band, and they then consider coefficient systems indexed on this category, with the split degree defined by using $m$ to view it as an object of $\catst$.


\section{Vanishing of cross-effects}\label{sec:cross-effects}

In this section, we give a general definition of the \emph{height} of a functor $\cC \to \cA$, for an abelian category $\cA$ and a category $\cC$ equipped with certain structure,\footnote{In fact, we give three definitions, each depending on a slightly different structure on $\cC$, and show that they agree whenever two are defined (Lemma \ref{lem:three-definitions}).} and relate it to various notions of \emph{height} appearing in the literature, including that of \cite{DjamentVespa2019FoncteursFaiblementPolynomiaux} (much of this section has been directly inspired by the definitions given in that paper). In particular, this encompasses the setting where $\cC$ is monoidal and its unit object is either initial or terminal (see \S\ref{para:specialise-DV} and \S\ref{para:specialise-HPV}), and also the setting where $\cC$ is any category equipped with a functor $\cI \to \cC$, where $\cI$ is the category defined just below at the beginning of \S\ref{para:first-def} (see \S\ref{para:specialise-this-paper}). In \S\ref{para:compare-two-heights} we study the intersection between these two settings. This is analogous to \S\ref{para:partial-braid-categories} above (which is concerned with the intersection between two different ways of defining the \emph{degree} of a functor with source $\cC$); see in particular Remark \ref{rmk:summary}.

Throughout this section $\cA$ will denote a fixed abelian category. In proofs we will often assume that $\cA$ is a category of modules over a ring, so that its objects have elements, which is justified by the Freyd-Mitchell embedding theorem.

\subsection{First definition.}\label{para:first-def}
Let $\cI$ be the category whose objects are the non-negative integers, and whose morphisms $m \to n$ are subsets of $\{ 1,\ldots,\mathrm{min}(m,n) \}$, with composition given by intersection. The endomorphism monoid $\mathrm{End}_{\cI}(n)$ is denoted $\I{n}$, and is the monoid of subsets of $\undn = \{1,\ldots,n\}$ under the operation $\cap$ with neutral element $\undn$ itself. We will also think of $\I{n}$ as a category on the single object $\bullet$. There is an operation $\mathrm{cr}(-)$ that takes a functor $f\colon \I{n} \to \cA$ as input and produces the following object of $\cA$:
\[
\mathrm{cr}(f) \;=\; \mathrm{im} \biggl(\, \sum_{S\subseteq\undn} (-1)^{\lvert S\rvert} f(\undn\smallsetminus S) \colon f(\bullet) \longrightarrow f(\bullet) \biggr)
\]
as output. Now suppose that we are given a category $\cC$ equipped, for each $n\geq 0$, with a collection of functors $\{ f_j\colon \I{n}\to\cC \}_{j\in J_n}$. Then the \emph{height} $\mathrm{ht}(T)\in\{-1,0,1,2,\ldots,\infty\}$ of a functor $T\colon\cC\to\cA$ is defined by the criterion that $\mathrm{ht}(T)\leq h$ if and only if for all $n>h$ and all $j\in J_n$, $\mathrm{cr}(T\circ f_j)=0$.

\subsection{Second definition.}\label{para:second-def}
Let $\J{n}$ denote the set of all subsets of $\undn$, considered as a partially-ordered set -- and thus as a category -- under the relation of inclusion of subsets. There is an operation $\crbar(-)$ taking a functor $f\colon \J{n}\to\cA$ as input and producing the following object of $\cA$:
\[
\crbar(f) \;=\; \mathrm{coker} \biggl(\, \bigoplus_{S \subsetneq \undn} f(S\hookrightarrow\undn) \colon \bigoplus_{S \subsetneq \undn} f(S) \longrightarrow f(\undn) \biggr)
\]
as output. Now suppose that we are given a category $\cC$ equipped, for each $n\geq 0$, with a collection of functors $\{ f_j\colon \J{n}\to\cC \}_{j\in J_n}$. Then the \emph{height} $\htbar(T)\in\{-1,0,1,2,\ldots,\infty\}$ of a functor $T\colon\cC\to\cA$ is defined by the criterion that $\htbar(T)\leq h$ if and only if for all $n>h$ and all $j\in J_n$, $\crbar(T\circ f_j)=0$.

\subsection{Relationship between the definitions.}\label{para:two-definitions}
There is a functor $z\colon \J{n}\to \I{n}$ given by sending each morphism $S\subseteq T$ in $\J{n}$ to the morphism $\undn\smallsetminus (T\smallsetminus S)$ in $\I{n}$. (More generally, any lattice $L$ may be viewed as a monoid $L^{\wedge}$ under the meet operation, and there is an analogous functor $L\to L^{\wedge}$ if $L$ is a Boolean algebra.) This relates the two constructions above as follows:

\begin{lem}\label{lem:two-definitions}
For any functor $f\colon \I{n}\to\cA$ the objects $\mathrm{cr}(f)$ and $\crbar(f\circ z)$ are isomorphic.
\end{lem}

A category $\cC$ equipped with collections of functors $\{\I{n}\to\cC\}$ as in the first definition may be viewed via $z$ as a category equipped with collections of functors $\{\J{n}\to\cC\}$ as in the second definition. Hence -- \emph{a priori} -- functors $T\colon\cC\to\cA$ have two possibly different heights, $\mathrm{ht}(T)$ and $\htbar(T)$. But the above lemma implies that these coincide, i.e.\ $\mathrm{ht}(T)=\htbar(T)$. The second definition is therefore more general, reducing to the first definition in the special case where the given functors $\J{n}\to\cC$ all factor through $z\colon \J{n}\to \I{n}$.

\begin{proof}[Proof of Lemma \ref{lem:two-definitions}]
This is proved exactly as Proposition 2.9 of \cite{DjamentVespa2019FoncteursFaiblementPolynomiaux}. We will give the details here, in order to identify (for later; see \S\ref{para:semi-functors}) where we use the fact that $f$ preserves the identity. First of all we will show that:
\begin{equation}\label{eq:ker-im-identity}
\mathrm{ker}(g) \;=\; \sum_{S\subsetneq\undn} \mathrm{im}(f(S)) \qquad\text{where}\qquad g = \displaystyle\sum_{S\subseteq\undn} (-1)^{\lvert S\rvert} f(\undn\smallsetminus S).
\end{equation}

\noindent $(\supseteq):$ Let $x=f(T)(y)$ for $y\in f(\bullet)$ and $T\subsetneq\undn$. Choose $i\in\undn\smallsetminus T$ and write
\[
g(x) \;\;= \sum_{S\subseteq \undn\smallsetminus\{i\}} \Bigl( (-1)^{\lvert S\rvert} f(\undn\smallsetminus S)f(T)(y) + (-1)^{\lvert S\rvert +1} f((\undn\smallsetminus S)\smallsetminus \{i\})f(T)(y) \Bigr) .
\]
Since $(T\smallsetminus S)\smallsetminus\{i\} = T\smallsetminus S$ the terms cancel pairwise and $x\in\mathrm{ker}(g)$.

\noindent $(\subseteq):$ Suppose $x\in f(\bullet)$ and $g(x)=0$. Since $f$ preserves the identity, i.e.\ $f(\undn)=\mathrm{id}$, we may write
\begin{align*}
x \;\;&= \sum_{\varnothing \neq S\subseteq\undn} (-1)^{\lvert S\rvert +1} f(\undn\smallsetminus S)(x) \\
&= \sum_{S\subsetneq\undn} (-1)^{n-\lvert S\rvert +1} f(S)(x) \quad \in \quad \sum_{S\subsetneq\undn} \mathrm{im}(f(S)).
\end{align*}

\noindent Now note that the right-hand side of (the left-hand equation of) \eqref{eq:ker-im-identity} is equal to the image of
\[
h = \bigoplus_{S\subsetneq\undn} f(S) \colon \bigoplus_{S\subsetneq\undn} f(\bullet) \longrightarrow f(\bullet).
\]
Hence we have $\mathrm{cr}(f) = \mathrm{im}(g) \cong f(\bullet)/\mathrm{ker}(g) = f(\bullet)/\mathrm{im}(h) = \mathrm{coker}(h) = \crbar(f\circ z)$.
\end{proof}

\subsection{Specialising to the setting of Djament and Vespa.}\label{para:specialise-DV}

Let $\cC$ be a symmetric monoidal category whose monoidal unit is null (simultaneously initial and terminal). In \cite{DjamentVespa2019FoncteursFaiblementPolynomiaux}, Djament and Vespa define the notion of a \emph{strong polynomial} functor $\cC\to\cA$ of \emph{degree} $d$. Their definition is recovered by the first definition of a \emph{functor of height $d$} above by equipping $\cC$ with the following collections of functors $\{\I{n}\to\cC\}$. Take $J_n$ to be the set of $n$-tuples $(x_1,\ldots,x_n)$ of objects of $\cC$. The associated functor $\I{n}\to\cC$ sends the unique object $\bullet$ to $\bigoplus_{i=1}^n x_i$ and a subset $S\subseteq\undn$ to the endomorphism $\bigoplus_{i=1}^n \phi_i$ where $\phi_i=\mathrm{id}$ for $i\in S$ and $\phi_i=0$ otherwise.

More generally, let $\cC$ be a symmetric monoidal category whose monoidal unit is initial. The general definition of Djament and Vespa is for this setting, and corresponds to the second definition of a \emph{functor of height $d$} above by equipping $\cC$ with the following collections of functors $\{\J{n}\to\cC\}$. Take $J_n$ to be the set of $n$-tuples $(x_1,\ldots,x_n)$ of objects of $\cC$ as before. The associated functor $\J{n}\to\cC$ sends the object $S\subseteq\undn$ to $\bigoplus_{i\in S}x_i$ and the inclusion $S\subseteq T$ to the canonical morphism $\bigoplus_{i\in S} x_i \cong \bigoplus_{i\in T} y_i \to \bigoplus_{i\in T} x_i$ where $y_i=x_i$ if $i\in S$ and $y_i=0$ otherwise.

Of course, our general definition of a \emph{functor of height $d$} introduced above specialises very naturally to this setting as it was directly inspired by the work of Djament and Vespa.\footnote{To see that it specialises as claimed to the setting of Djament and Vespa, combine D{\'e}finition 2.1, Proposition 2.3, D{\'e}finition 2.6 and Proposition 2.9 of \cite{DjamentVespa2019FoncteursFaiblementPolynomiaux}.} Soon we will generalise it slightly (\S\ref{para:semi-functors}) so that it also recovers the notion of \emph{height} used in \cite{Palmer2018Twistedhomologicalstability}. First we describe the dual of our second definition of height and specialise it to the setting of \cite{HartlPirashviliVespa2015Polynomialfunctorsalgebras}.

\subsection{Third definition.}\label{para:third-def}

There is an operation $\crbar^\prime(-)$ that takes a functor $f\colon \K{n}\to\cA$ as input and produces the following object of $\cA$:
\[
\crbar^\prime(f) \;=\; \mathrm{ker} \biggl(\, \bigoplus_{S \subsetneq \undn} f(S\hookrightarrow\undn) \colon f(\undn) \longrightarrow \bigoplus_{S \subsetneq \undn} f(S) \biggr)
\]
as output. Suppose that we are given a category $\cC$ equipped, for each $n\geq 0$, with a collection of functors $\{ f_j\colon \K{n}\to\cC \}_{j\in J_n}$. The \emph{height} $\htbar^\prime(T)\in\{-1,0,1,2,\ldots,\infty\}$ of a functor $T\colon\cC\to\cA$ is defined by the criterion that $\htbar^\prime(T)\leq h$ if and only if for all $n>h$ and all $j\in J_n$, $\crbar^\prime(T\circ f_j)=0$.

\subsection{Relation between all three definitions.}\label{para:three-definitions}

This may be related to the first and second definitions as follows. There is a functor $z^\prime\colon {\J{n}}^{\mathrm{op}}\to \I{n}$ given by sending each morphism $S\subseteq T$ in $\K{n}$ to the morphism $\undn\smallsetminus (T\smallsetminus S)$ in $\I{n}$. Using this and the functor $z\colon \J{n}\to \I{n}$ from above, any functor $\I{n}\to\cA$ induces functors $\J{n} \to \cA$ and $\K{n} \to \cA$.

\begin{lem}\label{lem:three-definitions}
For any functor $f\colon \I{n}\to\cA$ we have isomorphisms $\crbar(f\circ z) \cong \mathrm{cr}(f) \cong \crbar^\prime(f\circ z^\prime)$.
\end{lem}

A category $\cC$ equipped with collections of functors $\{\I{n}\to\cC\}$ as in the first definition may be viewed as a category equipped either with collections of functors $\{\J{n}\to\cC\}$ as in the second definition or collections of functors $\{\K{n}\to\cC\}$ as in the third definition. The above lemma implies that in this situation the three possible notions of height for functors $\cC\to\cA$ all coincide.

\begin{proof}[Proof of Lemma \ref{lem:three-definitions}]
By Lemma \ref{lem:two-definitions} it suffices to prove that $\crbar(f\circ z) \cong \crbar^\prime(f\circ z^\prime)$, in other words:
\begin{equation}\label{eq:coker-and-ker}
\mathrm{coker} \biggl( \bigoplus_{S\subsetneq\undn} f(S) \colon f(\bullet)^{n!-1} \longrightarrow f(\bullet) \biggr)
\;\cong\;
\mathrm{ker} \biggl( \bigoplus_{S\subsetneq\undn} f(S) \colon f(\bullet) \longrightarrow f(\bullet)^{n!-1} \biggr)
\end{equation}
where we have written $f(\bullet)^{n!-1}$ to denote $\bigoplus_{S\subsetneq\undn}f(\bullet)$. Since the morphisms $f(S)$ are idempotent and pairwise commute, there is a decomposition
\[
f(\bullet) \;\;\cong \bigoplus_{T\subseteq \cP^\prime(n)} \biggl( \bigcap_{S\in T} \mathrm{ker} (f(S)) \cap \bigcap_{S\in\cP^\prime(n)\smallsetminus T} \mathrm{im} (f(S)) \biggr) ,
\]
where $\cP^\prime(n)$ denotes the set of proper subsets of $\undn$.\footnote{\Cf the first part of the proof of Lemme 2.7 in \cite{CollinetDjamentGriffin2013Stabilitehomologiquepour}.} The direct sum of all components except the one corresponding to $T=\cP^\prime(n)$ is equal to $\sum_{S\subsetneq\undn} \mathrm{im}(f(S))$, so we have:
\begin{align*}
f(\bullet) \;&\cong \,\bigcap_{S\subsetneq\undn} \mathrm{ker}(f(S)) \;\oplus\; \sum_{S\subsetneq\undn} \mathrm{im}(f(S)) \\
&\cong \;\mathrm{ker} \biggl( \bigoplus_{S\subsetneq\undn} f(S) \colon f(\bullet) \longrightarrow f(\bullet)^{n!-1} \biggr) \;\oplus\; \mathrm{im} \biggl( \bigoplus_{S\subsetneq\undn} f(S) \colon f(\bullet)^{n!-1} \longrightarrow f(\bullet) \biggr) ,
\end{align*}
which implies the isomorphism \eqref{eq:coker-and-ker}, as desired.
\end{proof}

\subsection{Specialising to the setting of Hartl-Pirashvili-Vespa.}\label{para:specialise-HPV}

Let $\cC$ be a monoidal category whose monoidal unit is null, and which is not necessarily symmetric. In \cite{HartlPirashviliVespa2015Polynomialfunctorsalgebras}, Hartl, Pirashvili and Vespa define the notion of a \emph{polynomial} functor $\cC\to\cA$ of \emph{degree} $d$. (When $\cC$ is symmetric it agrees with the definition of \cite{DjamentVespa2019FoncteursFaiblementPolynomiaux}.) Their definition is recovered by our third definition of a \emph{functor of height $d$} by equipping $\cC$ with the following collections of functors $\{\K{n}\to\cC\}$.\footnote{See Definition 3.6 and Proposition 3.3 of \cite{HartlPirashviliVespa2015Polynomialfunctorsalgebras}.} As before, take $J_n$ to be the set of $n$-tuples $(x_1,\ldots,x_n)$ of objects of $\cC$. The associated functor $\K{n}\to\cC$ sends the object $S\subseteq\undn$ to the object $\bigoplus_{i\in S}x_i$ and the inclusion $S\subseteq T$ to the canonical morphism $\bigoplus_{i\in T} x_i \to \bigoplus_{i\in T} y_i \cong \bigoplus_{i\in S} x_i$ where $y_i=x_i$ if $i\in S$ and $y_i=0$ otherwise. Note: since $\cC$ is not symmetric, to correctly define $\bigoplus_{i\in S}x_i$ we must consider $\undn$ as a totally-ordered set and use the inherited ordering of each subset $S\subseteq\undn$.

We note that the definition of \cite{HartlPirashviliVespa2015Polynomialfunctorsalgebras} only requires the monoidal unit to be terminal. Also, the definition given earlier (\S \ref{para:specialise-DV}) for a symmetric monoidal category with initial unit works equally well when the monoidal structure is not symmetric, as long as one is careful, as in the previous paragraph, to use the natural total ordering on $\undn$. Thus there is a general notion of \emph{height} for functors $\cC\to\cA$ whenever $\cC$ is monoidal and its unit is either initial or terminal, and these notions coincide when the unit is null.

\subsection{Categories with finite coproducts; relation to the Taylor tower.}\label{para:finite-coproducts}

In \cite{HartlVespa2011Quadraticfunctorspointed} there is a definition of \emph{polynomial} functor $\cC\to\cA$ of \emph{degree} $d$ in the setting where $\cC$ has a null object and finite coproducts, and where $\cA$ is either $\mathsf{Ab}$ or $\mathsf{Grp}$, the category of groups. When $\cA=\mathsf{Ab}$ this is a special case of the definition of \cite{HartlPirashviliVespa2015Polynomialfunctorsalgebras}, since $\cC$ has a monoidal structure given by the coproduct. When $\cA=\mathsf{Grp}$ it falls outside the scope of the discussion in this section, since $\mathsf{Grp}$ is not an abelian category. It is, however, a \emph{semi-abelian category} (see \cite{JanelidzeMarkiTholen2002Semiabeliancategories,Borceux2004surveysemiabelian}), which suggests that it would be interesting to try to extend the general notion of the \emph{height} of a functor $\cC\to\cA$ in this section to the case where $\cA$ is only semi-abelian (for example the category of groups or the category of non-unital rings).

As an aside, we recall that when the monoidal structure on $\cC$ is given by the coproduct, one can do more than just define the height of a functor $T \colon \cC \to \cA$: one can also approximate it by functors of smaller height, and these approximations form its so-called \emph{Taylor tower}. The key property of the coproduct that allows this is that its universal property equips us with ``fold maps'' $c + \cdots + c \to c$. In the next paragraph, we recall briefly the construction from \cite{HartlVespa2011Quadraticfunctorspointed}, using the terminology of the present section. 

Recall that the structure on $\cC$ used to define the height of a functor defined on it is a collection of functors $f_{(c_1,\ldots,c_n)} \colon \K{n} \to \cC$, one for each $n$-tuple of of objects in $\cC$ (and for each $n\geq 0$), and the cross-effect $\crbar^\prime(Tf_{(c_1,\ldots,c_n)})$ is a subobject of $T(c_1 + \cdots + c_n)$, where $+$ denotes the coproduct in $\cC$. Now take $c_1 = \cdots = c_n = c$. The universal property of the coproduct gives us a morphism $c + \cdots + c \to c$, to which we may apply $T$ and then compose with the inclusion of the cross-effect to obtain a morphism $\crbar^\prime(Tf_{(c,\ldots,c)}) \to T(c)$. Define $p_{n-1} T(c)$ to be the cokernel of this morphism. This construction is functorial in $c$ and defines a functor $p_{n-1} T$ of height $\leq n-1$, which is to be thought of as the best approximation of $T$ by such a functor. There are also natural transformations $p_0 T \leftarrow p_1 T \leftarrow \cdots \leftarrow p_{n-1} T \leftarrow p_n T \leftarrow \cdots$ and $T \to \mathrm{lim}(p_\bullet T)$, which between them constitute the ``Taylor tower'' of $T$.

\subsection{Specialising to the setting of Collinet-Djament-Griffin.}\label{para:specialise-CDG}

Let $\Se$ denote the category of finite sets and partially-defined functions and let $\Sigma$ denote its subcategory of finite sets and partially-defined injections. For any intermediate category $\Sigma \leq \Lambda \leq \Se$ and any category $\cC$ we may define the \emph{wreath product} $\cC \wr \Lambda$ to have finite tuples of objects of $\cC$ as objects, and a morphism from $(c_1,\ldots,c_m)$ to $(d_1,\ldots,d_n)$ to consist of a morphism $\phi\colon m \to n$ of $\Lambda$ together with morphisms $\alpha_i \colon c_i \to d_{\phi(i)}$ of $\cC$ for each $i$ on which $\phi$ is defined. We write this morphism as $(\phi \mathbin{;} \{\alpha_i\}_{i\in \mathrm{dom}(\phi)})$.

We may then equip $\cC \wr \Lambda$ with collections of functors $\{ \I{n} \to \cC \wr \Lambda \}$, as follows. As before, take the indexing set $J_n$ to be the set of $n$-tuples of objects of $\cC$. The functor $\I{n} \to \cC \wr \Lambda$ associated to the $n$-tuple $(c_1,\ldots,c_m)$ takes the unique object $\bullet$ of $\I{n}$ to $(c_1,\ldots,c_m)$ and a subset $S \subseteq \undn$ to the endomorphism $(r_S \mathbin{;} \{\mathrm{id}_{c_i}\}_{i\in S})$ where $r_S(i)=i$ for $i\in S$ and $r_S(i)$ is undefined otherwise.

This defines a notion of \emph{height} for any functor $T \colon \cC \wr \Lambda \to \cA$ into an abelian category $\cA$, using the first definition (\S\ref{para:first-def}) above. This exactly recovers the definition of \emph{height} given by Collinet, Djament and Griffin~\cite{CollinetDjamentGriffin2013Stabilitehomologiquepour} in this setting. To see this, we may by Lemma \ref{lem:three-definitions} use the third definition (\S\ref{para:third-def}) above instead. Unravelling this definition, we see that it is precisely the definition of \cite{CollinetDjamentGriffin2013Stabilitehomologiquepour}, given in D{\'e}finitions 2.5 together with the sentence before Proposition 2.11.\footnote{A small difference is that they additionally assume that $T(\varnothing)$ is the zero object of $\cA$. So, for example, a functor $\cC \wr \Lambda \to \cA$ taking every object to a fixed object $a \neq 0$ of $\cA$ and every morphism to $\mathrm{id}_a$ has height zero according to our definition, whereas it does not have any finite height according to the definition of \cite{CollinetDjamentGriffin2013Stabilitehomologiquepour}. The difference is analogous to the difference between linear and affine functions.}

\subsection{Semi-functors.}\label{para:semi-functors}

The construction in \S\ref{para:first-def} taking a functor $f\colon \I{n}\to\cA$ as input and returning an object $\mathrm{cr}(f)$ of $\cA$ works also if $f$ is just a \emph{semi-functor}, in other words preserving composition but not necessarily identities.\footnote{In fact, for this construction, there is no need even for it to preserve composition -- but we will want this later.} So if $\cC$ is a category equipped, for each $n\geq 0$, with a collection of semi-functors $\{ f_j\colon \I{n}\to\cC \}_{j\in J_n}$, then we may define the \emph{height} of a semi-functor $T\colon\cC\to\cA$ exactly as before: $\mathrm{ht}(T)\leq h$ if and only if for all $n>h$ and all $j\in J_n$, $\mathrm{cr}(T\circ f_j)=0$. The second (\S\ref{para:second-def}) and third (\S\ref{para:third-def}) definitions of height generalise in the same way: if $\cC$ is a category equipped with collections of semi-functors $\{\J{n}\to\cC\}_{j\in J_n}$ or $\{\K{n} \to\cC\}_{j\in J_n}$ then we have a notion of the \emph{height} of any semi-functor $\cC\to\cA$, defined exactly as in the case of functors.

Lemma \ref{lem:two-definitions} is no longer true for semi-functors: the fact that $f$ preserves the identity was used to prove one of the two inclusions for the equality \eqref{eq:ker-im-identity}. However, the rest of the proof goes through and shows that there is an exact sequence $\crbar(f\circ z) \to \mathrm{cr}(f) \to 0$ in this setting. The proof of Lemma \ref{lem:three-definitions} does not use that $f$ preserves the identity, so we still have that $\crbar(f\circ z) \cong \crbar^\prime(f\circ z^\prime)$ when $f$ is a semi-functor. As a result, if $\cC$ is a category equipped with collections of semi-functors $\{ \I{n}\to\cC \}$ and $T\colon \cC\to\cA$ is a semi-functor, then:
\[
\mathrm{ht}(T) \leq \htbar(T) = \htbar^\prime(T).
\]
In fact, $\htbar(T)$ is often infinite when the semi-functors $\{\J{n}\to\cC\}$ are not functors (\cf Proposition \ref{prop:htbar-is-infinite}) so the right notion in this case is $\mathrm{ht}(T)$, which we will use in the next subsection.

\subsection{Specialising to partial braid categories.}\label{para:specialise-this-paper}

As before, we denote by $\cI$ the category with objects $0,1,2,\ldots$ and morphisms $m \to n$ corresponding to subsets of $\{ 1,\ldots,\mathrm{min}(m,n) \}$, with composition given by intersection. We will sometimes think of these morphisms $m \to n$ as partially-defined functions $\{1,\ldots,m\} \to \{1,\ldots,n\}$ that are the identity wherever they are defined. We will usually abbreviate $\{1,\ldots,n\}$ as $\undn$.

Let $\cC$ be a category equipped with a functor $s\colon \cI \to \cC$. For example, $\cC$ could be an object of $\cati$, in the notation of \S\ref{sss:height} below, which in particular includes the case where $\cC$ is the partial braid category $\cB(M,X)$ defined in \S\ref{sss:some-functors} (see also \S\ref{para:degree-WRW}).

Now equip $\cC$ with the following collections of semi-functors $\{ f_m\colon \I{n}\to\cC \}_{m\in J_n}$. Take the indexing set $J_n$ to be $\bN\cap [n,\infty)$. Then for $m\geq n$ let $f_m$ be the composite semi-functor
\[
\I{n} \to \I{m} = \mathrm{End}_{\cI}(m) \hookrightarrow \cI \xrightarrow{\;s\;} \cC,
\]
where $\I{n}\to \I{m}$ takes a subset $S$ of $\undn$ to the subset $S+m-n = \{s+m-n \mid s\in S\}$ of $\undm$. This defines a notion of \emph{height} for each semi-functor $T\colon\cC\to\cA$. Unwinding the definition, it says that $\mathrm{ht}(T)\leq h$ if and only if for all $m\geq n>h$ the following subobject of $Ts(m)$ vanishes:
\begin{equation}\label{eq:subobject-of-Tsm}
\mathrm{im} \biggl(\, \sum_{S\subseteq\{m-n+1,\ldots,m\}} (-1)^{\lvert S\rvert} Ts(f_{S\cup\{1,\ldots,m-n\}}) \biggr) .
\end{equation}
Here $f_T\colon \undm\to\undm$ is the partially-defined function that ``forgets'' $T\subseteq\undm$, in other words $f_T(i)=i$ if $i\in\undm\smallsetminus T$ and is undefined if $i\in T$.

\begin{lem}\label{lem:isomorphism-of-subobjects}
The subobject \eqref{eq:subobject-of-Tsm} of $Ts(m)$ is equal to the subobject
\begin{equation}\label{eq:subobject-of-Tsm-2}
\mathrm{im}(Ts(f_{\{1,\ldots,m-n\}})) \cap \bigcap_{i=m-n+1}^m \mathrm{ker}(Ts(f_{\{i\}})).
\end{equation}
\end{lem}

As a corollary, we deduce that the definition of \emph{height} used in the paper \cite{Palmer2018Twistedhomologicalstability} (see Definition 3.15 of that paper) is recovered when we specialise in this way, taking $\cC = \cB(M,X)$ equipped with the canonical functor $\cI \to \cB(M,X)$ (see the paragraph below \eqref{eq:partial-braid-functor-v2}).

\begin{coro}
The height of a functor $\cB(M,X)\to\mathsf{Ab}$ given by \textup{Definition 3.15} of \textup{\cite{Palmer2018Twistedhomologicalstability}} agrees with the definition above, specialised to the case $\cC=\cB(M,X)$ and $\cA=\mathsf{Ab}$.
\end{coro}

\begin{proof}
By Definition 3.15 of \cite{Palmer2018Twistedhomologicalstability}, the height of a functor $T\colon\cB(M,X)\to\mathsf{Ab}$ is bounded above by $h$ if and only if for all $m\geq n>h$ we have $T_m^n=0$. Looking at the definition of $T_m^n$ (see Definition 3.10 of \cite{Palmer2018Twistedhomologicalstability}) we see that it is precisely \eqref{eq:subobject-of-Tsm-2}, and therefore \eqref{eq:subobject-of-Tsm}, by Lemma \ref{lem:isomorphism-of-subobjects}.
\end{proof}

\begin{proof}[Proof of Lemma \ref{lem:isomorphism-of-subobjects}]
We think of $\cA$ as a concrete category of modules over a ring, by the Freyd-Mitchell embedding theorem, so that we may talk about the elements of its objects.

\noindent $\bullet\; \eqref{eq:subobject-of-Tsm-2} \subseteq \eqref{eq:subobject-of-Tsm}:$ Suppose that $x$ is an element of \eqref{eq:subobject-of-Tsm-2}, say $x=Ts(f_{\{1,\ldots,m-n\}})(y)$. If $S$ is a non-empty subset of $\{m-n+1,\ldots,m\}$ then we may pick some $i\in S$ and compute that
\begin{align*}
Ts(f_{S\cup\{1,\ldots,m-n\}})(y) &= Ts(f_S) \circ Ts(f_{\{i\}}) \circ Ts(f_{\{1,\ldots,m-n\}})(y) \\
&= Ts(f_S) \circ Ts(f_{\{i\}}) (x) = 0.
\end{align*}
Hence we deduce that
\[
\sum_{S\subseteq\{m-n+1,\ldots,m\}} (-1)^{\lvert S\rvert} Ts(f_{S\cup\{1,\ldots,m-n\}}) (y) = x,
\]
so in particular $x\in\eqref{eq:subobject-of-Tsm}$.

\noindent $\bullet\; \eqref{eq:subobject-of-Tsm} \subseteq \eqref{eq:subobject-of-Tsm-2}:$ Now suppose that we begin with an element $x$ of the form
\[
x = \sum_{S\subseteq\{m-n+1,\ldots,m\}} (-1)^{\lvert S\rvert} Ts(f_{S\cup\{1,\ldots,m-n\}}) (y).
\]
Then for $m-n+1\leq i\leq m$ we have
\begin{align*}
Ts(f_{\{i\}})(x) = \sum_{S\subseteq\{m-n+1,\ldots,m\}\smallsetminus\{i\}} &(-1)^{\lvert S\rvert} Ts(f_{\{i\}}) \circ Ts(f_{S\cup\{1,\ldots,m-n\}})(y) \\
&+ (-1)^{\lvert S\rvert +1} Ts(f_{\{i\}}) \circ Ts(f_{S\cup\{i\}\cup\{1,\ldots,m-n\}}) (y)
\end{align*}
which vanishes since the terms pairwise cancel, so $x\in\mathrm{ker}(Ts(f_{\{i\}}))$. One can similarly show that $Ts(f_{\{1,\ldots,m-n\}})(x)=x$, so $x\in\mathrm{im}(Ts(f_{\{1,\ldots,m-n\}}))$.
\end{proof}

\subsection{Two notions of height on a cyclic monoidal category.}\label{para:compare-two-heights}

There is an overlap between the definition in \S\ref{para:specialise-DV} of the height of a functor $\cC\to\cA$ when $\cC$ is equipped with a monoidal structure\footnote{In \S\ref{para:specialise-DV} it was assumed that the monoidal structure is symmetric, but, as remarked in \S\ref{para:specialise-HPV}, the symmetry is not really necessary for the definition.} with null unit and the definition in \S\ref{para:specialise-this-paper} of the height of a functor $\cC\to\cA$ when $\cC$ is equipped with a functor $\cI\to\cC$.

Recall that $\cI$ and $\Sigma$ have objects $0,1,2,\ldots$, morphisms $m \to n$ of $\Sigma$ are partially-defined injections $\undm \to \undn$ and morphisms of $\cI$ are those partially-defined injections that are the identity wherever they are defined. Let $\cB$ have the same objects and take the morphisms $m \to n$ of $\cB$ to be partially-defined braided injections from $\undm$ to $\undn$. In other words, it is the category $\cB(\bR^2)$ from \S\ref{sss:some-functors}. There is an embedding $\cI \subset \cB$ and a functor $\cB \to \Sigma$ which compose to an embedding $\cI \subset \Sigma$.

Now let $\cC$ be a strict monoidal category and pick an object $x$ of $\cC$. There is then a natural functor $s \colon \cI \to \cC$ that takes $n$ to $x^{\oplus n}$. If $\cC$ is braided then $s$ extends to a monoidal functor $\cB \to \cC$ and if it is symmetric then $s$ extends further to a monoidal functor $\Sigma \to \cC$.\footnote{In the notation of \S\ref{sss:some-functors}, $\Sigma$ is $\cB(\bR^\infty)$ and $\cB$ is $\cB(\bR^2)$, whereas $\cI$ is a (non-monidal) subcategory of $\cB(\bR)$. For any monoidal category $\cC$ and object $x$ of $\cC$, there is a unique monoidal functor $\cB(\bR) \to \cC$ sending $1$ to $x$; its restriction to $\cI \subset \cB(\bR)$ is the ``natural'' functor $s$ to which we are referring.}

Now assume that the unit object of $\cC$ is null and that $\cC$ is \emph{generated} by $x$ in the sense that every object of $\cC$ is isomorphic to $x^{\oplus n}$ for some (not necessarily unique) non-negative integer $n$. In this sense we may say that $\cC$ is a \emph{cyclic monoidal category}. For example, if the manifold $M$ splits as $\bR \times N$ for some manifold $N$, then the category $\cB(M,X)$ defined in \S\ref{sss:some-functors} is a cyclic monoidal category generated by the object $1$. The natural functor $s \colon \cI \to \cB(M,X)$ taking $1$ to $1$ is exactly the one constructed in the paragraph below \eqref{eq:partial-braid-functor-v2}. If $N$ splits further as $\bR \times N^\prime$ then $\cB(M,X)$ is braided, and if $N = \bR^2 \times N^{\prime\prime}$ then it is symmetric. One may then define the \emph{height} of a functor $T\colon\cC\to\cA$ either using the monoidal structure of $\cC$ as in \S\ref{para:specialise-DV} -- this will be denoted $\mathrm{ht}_\oplus(T)$ -- or using the functor $s \colon \cI \to \cC$ as in \S\ref{para:specialise-this-paper} -- this will be denoted $\mathrm{ht}_\cI(T)$.

\begin{prop}\label{prop:compare-two-heights-special-case}
For any functor $T \colon \cC \to \cA$ we have $\mathrm{ht}_\cI(T) \leq \mathrm{ht}_\oplus(T)$. If we assume that $\cC$ is braided we have an equality $\mathrm{ht}_\cI(T) = \mathrm{ht}_\oplus(T)$.
\end{prop}

We will prove this as a corollary of a slightly more general setup.

\begin{defn}\label{def:o-s-partition}
Fix $m,n\geq 0$. An \emph{ordered shifted partition} $\lambda$ of $m$ of \emph{length} $n$ -- written $\lambda\vdash m$ and $\lvert\lambda\rvert = n$ -- is an ordered $(n+1)$-tuple $(\lambda_0,\lambda_1,\ldots,\lambda_n)$ of non-negative integers whose sum is $m$. Associated to this there is a semigroup homomorphism $\psi_\lambda \colon \I{n} \to \I{m}$ taking a subset $S$ of $\undn$ to the subset $S_\lambda$ of $\undm$, where $S_\lambda$ is defined as follows:
\[
S_\lambda = \bigcup_{i\in S}\{i\}_\lambda \qquad\qquad \{i\}_\lambda = \{ \lambda_0 + \cdots + \lambda_{i-1} + 1, \ldots, \lambda_0 + \cdots + \lambda_i \}.
\]
We are viewing $\cI_n$ as a monoid under intersection, with identity element $\undn$, so $\psi_\lambda$ is a monoid homomorphism if and only if $\lambda_0 = 0$.
\end{defn}

\begin{defn}\label{def:two-types-of-degree}
Now let $\cC$ be any category and $s \colon \cI \to \cC$ a functor. We obtain (semi-)functors $f_\lambda\colon \I{n}\to\cC$ defined by
\[
\I{n} \xrightarrow{\, \psi_\lambda \,} \I{m} = \mathrm{End}_\cI(m) \hookrightarrow \cI \xrightarrow{\; s\;} \cC.
\]
For any functor $T \colon \cC \to \cA$, we define $\mathrm{ht}_\cI(T)$ and $\mathrm{ht}_\oplus(T)$ by the condition that (for $\square = \cI \text{ or } {\oplus}$) $\mathrm{ht}_\square(T)\leq h$ if and only if for each ordered shifted partition $\lambda\vdash m$ of length $\lvert\lambda\rvert >h$,
\begin{itemizeb}
\item[$(\square = \oplus)$] with $\lambda_0=0$,
\item[$(\square = \cI)$] with $\lambda_1 = \cdots = \lambda_n = 1$,
\end{itemizeb}
the cross-effect $\mathrm{cr}(Tf_\lambda)$ vanishes. We similarly define $\htbar_\square(T)$ using $\crbar(Tf_\lambda z)$ in place of $\mathrm{cr}(Tf_\lambda)$. In other words, when $\square = \cI$ we define the height using (for each $n\geq 0$) the collection of semi-functors $\{f_\lambda \colon \I{n} \to \cC \mid \lambda \vdash m, \lvert\lambda\rvert = n, \lambda_1 = \cdots = \lambda_n = 1 \}$ and when $\square = \oplus$ we define the height using the collection of functors $\{f_\lambda \colon \I{n} \to \cC \mid \lambda \vdash m, \lvert\lambda\rvert = n, \lambda_0 = 0 \}$.
\end{defn}

In the previous setup, with a cyclic monoidal category $\cC$ generated by the object $x$, we had a natural functor $s \colon \cI \to \cC$ taking $1$ to $x$. For a functor $T \colon \cC \to \cA$ we defined $\mathrm{ht}_\cI(T)$ to be the height of $T$ as defined in \S\ref{para:specialise-this-paper}, using the structure given by the functor $s$. This is exactly the same as the definition of $\mathrm{ht}_\cI(T)$ given in Definition \ref{def:two-types-of-degree}. Moreover, we defined $\mathrm{ht}_\oplus(T)$ to be the height of $T$ as defined in \S\ref{para:specialise-DV}, using the monoidal structure of $\cC$. Unravelling the definitions, one can see that this is exactly the same as the definition of $\mathrm{ht}_\oplus(T)$ given in Definition \ref{def:two-types-of-degree}, using just the functor $s \colon \cI \to \cC$. (For this fact, it is critical that $\cC$ is generated by the object $x$.)

Thus Definition \ref{def:two-types-of-degree} for a category $\cC$ equipped with a functor $s \colon \cI \to \cC$ generalises the setting described at the beginning of this subsection, where the functor $s$ arose from the structure of $\cC$ as a cyclic monoidal category.

For the rest of this subsection, unless otherwise stated, we assume that we are in the general setting of a category $\cC$ equipped with a functor $s \colon \cI \to \cC$, and we use the definitions of height from Definition \ref{def:two-types-of-degree}.

\begin{rmk}\label{rmk:add-assumption-to-defn}
It is not hard to see that if $\lambda_i=0$ for some $i\geq 1$ then $\mathrm{cr}(Tf_\lambda)=0$. Thus, if $\lambda_0=0$ too (so that $f_\lambda$ is a functor and Lemma \ref{lem:two-definitions} applies), we have $\crbar(Tf_\lambda z)=0$. So in Definition \ref{def:two-types-of-degree}, when $\square = \oplus$, we may assume that $\lambda_1,\ldots,\lambda_n \geq 1$.
\end{rmk}

\begin{rmk}
When $\square=\oplus$ the $f_\lambda$ involved in the definition are all \emph{functors}, so $\mathrm{ht}_\oplus(T) = \htbar_\oplus(T)$, by the discussion following Lemma \ref{lem:two-definitions}. When $\square=\cI$ we only know that $\mathrm{ht}_\cI(T) \leq \htbar_\cI(T)$, as discussed in \S\ref{para:semi-functors}. We first show that $\htbar_\cI(T)$ is in fact almost always infinite.
\end{rmk}

\begin{prop}\label{prop:htbar-is-infinite}
Let $T\colon\cC\to\cA$ be any functor. Then $\htbar_\cI(T)>0$ implies that $\htbar_\cI(T)=\infty$.
\end{prop}

So it remains to compare $\mathrm{ht}_\cI(T)$ and $\mathrm{ht}_\oplus(T)$, which we do after the next definition.

\begin{defn}\label{def:admits-conjugations}
We say that a functor $T\colon\cC\to\cA$ \emph{admits conjugations} if the composite functor $Ts\colon\cI\to\cA$ extends to some category $\cS\supset\cI$ and for each $n\geq 0$ and $R,S\subseteq\undn$ with $\lvert R\rvert = \lvert S\rvert$ there exists an automorphism $\phi \in \mathrm{Aut}_\cS(n)$ such that $\phi r_R \phi^{-1} = r_S$, where $r_R\in\mathrm{End}_\cI(n)$ denotes the endomorphism that restricts to $R$, i.e., is the identity on $R$ and undefined on $\undn\smallsetminus R$.
\end{defn}

\begin{eg}\label{eg:admitting-conjugations-braiding}
If $\cC$ is a strict \emph{braided} monoidal category with null unit object, generated by the object $x$, then the natural functor $s \colon \cI \to \cC$ extends to $\cB \supset \cI$, as explained at the beginning of this subsection. In this case \emph{every} functor $\cC \to \cA$ admits conjugations: we may take $\cS=\cB$ and for $\phi$ choose any braid connecting the points $R$ with the points $S$ and the points $\undn\smallsetminus R$ with the points $\undn\smallsetminus S$. In particular, this applies to $\cC = \cB(M,X)$ as defined in \S\ref{sss:some-functors} when $M$ is of the form $\bR^2 \times N$.
\end{eg}

\begin{eg}\label{eg:admitting-conjugations-BMX}
In fact, for any $M$ (of dimension at least two), if we take $\cC = \cB(M,X)$ with the natural functor $s \colon \cI \to \cB(M,X)$ (\cf \eqref{eq:partial-braid-functor-v2}), then every functor $\cC \to \cA$ admits conjugations: we may take $\cS$ to be $\cB(M,X)$ itself and for $\phi$ choose any braid on $M$ that connects the points $\{a_i \mid i\in R\}$ with the points $\{a_i \mid i\in S\}$ and the points $\{a_i \mid i\in\undn\smallsetminus R\}$ with the points $\{a_i \mid i\in\undn\smallsetminus S\}$.
\end{eg}

\begin{prop}\label{prop:compare-two-heights}
For any functor $T\colon\cC\to\cA$ we have $\mathrm{ht}_\cI(T) \leq \mathrm{ht}_\oplus(T)$. If $T$ admits conjugations then $\mathrm{ht}_\cI(T) = \mathrm{ht}_\oplus(T)$. However, in general the inequality may be strict: for any $h\in\{2,\ldots,\infty\}$ there exists a functor $T_h\colon \cI\to\mathsf{Ab}$ such that $\mathrm{ht}_\cI(T_h)=1$ but $\mathrm{ht}_\oplus(T_h)=h$.
\end{prop}

\begin{proof}[Proof of Proposition \ref{prop:compare-two-heights-special-case}]
This now follows from Proposition \ref{prop:compare-two-heights} and Example \ref{eg:admitting-conjugations-braiding}.
\end{proof}

\begin{rmk}\label{rmk:compare-heights-for-BMX}
Proposition \ref{prop:compare-two-heights-special-case} applied to the cyclic monoidal category $\cC = \cB(\bR \times N,X)$ tells us that $\mathrm{ht}_\cI(-) \leq \mathrm{ht}_\oplus(-)$ with equality if $N = \bR \times N^\prime$. But by Proposition \ref{prop:compare-two-heights} and Example \ref{eg:admitting-conjugations-BMX} we know that in fact $\mathrm{ht}_\cI(-) = \mathrm{ht}_\oplus(-)$ for $\cC = \cB(M,X)$ for \emph{any} $M$ (of dimension at least two). This suggests that it should be possible to generalise Proposition \ref{prop:compare-two-heights-special-case} to a setting where $\cC$ is a left module over a cyclic monoidal category, analogously to Proposition \ref{p:two-degrees-agree-2} for degree.
\end{rmk}

\begin{rmk}[\textit{Summary.}]\label{rmk:summary}
For a functor $T \colon \cC \to \cA$ we have the following square of equalities:
\begin{equation}\label{eq:square-of-equalities}
\centering
\begin{split}
\begin{tikzpicture}
[x=1mm,y=1mm]
\node at (0,12) {$\mathrm{deg}^{x}(T)$};
\node at (24,12) {$\mathrm{deg}(T)$};
\node at (0,0) {$\mathrm{ht}_{\cI}(T)$};
\node at (24,0) {$\mathrm{ht}_{\oplus}(T)$};
\node at (12,12) {$=$};
\node at (12,14.5) [font=\footnotesize] {(a)};
\node at (12,0) {$=$};
\node at (12,-2.5) [font=\footnotesize] {(b)};
\node at (0,6) {\rotatebox{90}{$=$}};
\node at (-3,6) [font=\footnotesize] {(c)};
\node at (24,6) {\rotatebox{90}{$=$}};
\node at (27,6) [font=\footnotesize] {(d)};
\end{tikzpicture}
\end{split}
\end{equation}
(using notation of \S\ref{para:partial-braid-categories} in the top row), where
\begin{itemizeb}
\item[(a)] holds when $\cC$ is a left-module over a braided monoidal category with null unit and generating object $x$ (Proposition \ref{p:two-degrees-agree-2});
\item[(b)] holds when $\cC$ is a braided monoidal category with null unit and generating object $x$ (Proposition \ref{prop:compare-two-heights-special-case}) or $\cC = \cB(M,X)$ (Remark \ref{rmk:compare-heights-for-BMX});
\item[(c)] holds when $\cC = \cB(M,X)$, by Lemma 3.16 of \cite{Palmer2018Twistedhomologicalstability}
\item[(d)] holds when $\cC = \cB(\bR^3 \times N,X)$, by Proposition 2.3 of \cite{DjamentVespa2019FoncteursFaiblementPolynomiaux} --- more generally, they prove this for $\cC$ a symmetric monoidal category with initial unit.
\end{itemizeb}
Putting these together, we see that in fact (d) holds whenever $\cC = \cB(M,X)$ for \emph{any} $M$, via (a)--(c). Also, (c) holds whenever $\cC$ is a symmetric monoidal category with null unit and generating object $x$, via (a), (b) and (d).

This suggests that (c) should generalise to $\cC$ any left-module over a braided monoidal category with null unit and generating object $x$ and (d) should generalise to $\cC$ any left-module over a braided monoidal category with initial unit. This would imply that (b) also generalises to $\cC$ any left-module over a braided monoidal category with null unit and generating object $x$ (\cf Remark \ref{rmk:compare-heights-for-BMX}).
\end{rmk}

In the remainder of this subsection we prove Propositions \ref{prop:htbar-is-infinite} and \ref{prop:compare-two-heights}.

\begin{proof}[Proof of Proposition \ref{prop:htbar-is-infinite}]
Let us abbreviate $Ts(n)$ to $T_n$ and for a subset $S\subseteq\{1,\ldots,\mathrm{min}(k,l)\}$ let us write simply $r_S\colon T_k \to T_l$ instead of $Ts(r_S)$. (Recall that $r_S\colon \underline{k} \to \underline{l}$ is the identity on $S$ and undefined elsewhere.)

Now suppose that $\htbar_\cI(T)<\infty$. In particular this implies that, for some $h\geq 0$ and any $k\geq 0$,
\[
T_{k+h} = \sum_{S\subsetneq\underline{h}} \mathrm{im}(r_{S+k}).
\]
But each $\mathrm{im}(r_{S+k})$ is contained in $\mathrm{im}(r_{\{k+1,\ldots,k+h\}})$, so $r_{\{k+1,\ldots,k+h\}} \colon T_{k+h} \to T_{k+h}$
is surjective. Also note that $r_{\underline{k}} \colon T_{k+h} \to T_k$ is surjective, since it has a right-inverse. The commutative square
\begin{equation}
\centering
\begin{split}
\begin{tikzpicture}
[x=1mm,y=1mm]
\node (tm) at (0,15) {$T_{k+h}$};
\node (tr) at (40,15) {$T_k$};
\node (bm) at (0,0) {$T_{k+h}$};
\node (br) at (40,0) {$T_k$};
\draw[->>] (tm) to node[above,font=\small]{$r_{\underline{k}}$} (tr);
\draw[->>] (bm) to node[left,font=\small]{$r_{\{k+1,\ldots,k+h\}}$} (tm);
\draw[->] (bm) to node[above,font=\small]{$r_\varnothing$} (br);
\draw[->] (br) to node[right,font=\small]{$r_\varnothing$} (tr);
\draw[draw=none,use as bounding box](-20,0) rectangle (60,18);
\end{tikzpicture}
\end{split}
\end{equation}
therefore tells us that $r_\varnothing \colon T_k \to T_k$ is also surjective. So for any $m\geq n>0$ we have
\[
T_m = \mathrm{im}(r_\varnothing) = \sum_{S\subsetneq\undn} \mathrm{im}(r_{S+m-n}),
\]
and so $\htbar_\cI(T)\leq 0$.
\end{proof}

\begin{proof}[Proof of Proposition \ref{prop:compare-two-heights}]
We first prove the inequality $\mathrm{ht}_\cI(T)\leq\mathrm{ht}_\oplus(T)$, then give the promised example of $T$ where it is strict, and then finally show that the additional assumption that $T$ admits conjugations rules out this possibility, i.e., that $\mathrm{ht}_\cI(T)=\mathrm{ht}_\oplus(T)$ for such $T$.

\vspace{1ex}
\noindent (a) \textit{Proof of the inequality.}
We use the notation of the previous proof, abbreviating $Ts(r_S)$ to $r_S$. Let $m\geq n>\mathrm{ht}_\oplus(T)$. We need to show that
\begin{equation}\label{eq:proof-of-inequality}
\sum_{S\subseteq\undn} (-1)^{\lvert S\rvert} r_{(\undn \smallsetminus S)+m-n}
\end{equation}
is the zero morphism. Let $\lambda$ be the ordered shifted partition of length $n$ with $\lambda_0=0$, $\lambda_1=m-n+1$ and $\lambda_i=1$ otherwise. Then \eqref{eq:proof-of-inequality} is equal to
\[
\sum_{S\subseteq\undn} (-1)^{\lvert S\rvert} r_{\{m-n+1,\ldots,m\}} \circ r_{S_\lambda} \;=\; r_{\{m-n+1,\ldots,m\}} \circ \biggl( \sum_{S\subseteq\undn} (-1)^{\lvert S\rvert} \circ r_{S_\lambda} \biggr) .
\]
Since $\lvert \lambda \rvert = n > \mathrm{ht}_\oplus(T)$, the morphism in brackets on the right-hand side is zero, and so \eqref{eq:proof-of-inequality} is zero, as required.\hfill $\diamond$

\vspace{1ex}
\noindent (b) \textit{Example of strict inequality.}
For this example we will take $\cC$ to be $\cI$ itself, with $s=\mathrm{id}\colon\cI\to\cI$. Fix $h\in\{2,3,4,\ldots,\infty\}$ and define a functor $T_h\colon \cI\to\mathsf{Ab}$ as follows. The object $n$ is taken to the free abelian group
\[
\bZ\{ S\subseteq\undn \mid \lvert S\rvert \leq h \text{ and } S \text{ has no consecutive elements} \}
\]
and for $R\subseteq\{1,\ldots,\mathrm{min}(m,n)\}$ the morphism $r_R\colon \undm \to \undn$ is taken to the homomorphism $T_h(m)\to T_h(n)$ that sends the basis element $S\subseteq\undm$ to the basis element $r_R(S) = S\cap R \subseteq \undn$.

This will turn out to have $\mathrm{ht}_\cI(T_h)=1<h=\mathrm{ht}_\oplus(T_h)$. The idea is that both $\mathrm{ht}_\cI(-)$ and $\mathrm{ht}_\oplus(-)$ examine a functor $T$ using certain partitions -- but the former only uses partitions in which each piece has size $1$ and is therefore sensitive to ``interference'' from the condition above that subsets have \emph{no consecutive elements} and therefore measures the ``wrong'' height, whereas the latter uses partitions with pieces of arbitrary size, and so is insensitive to such interference.

To show that $\mathrm{ht}_\oplus(T_h)=h$ we take $\lambda\vdash m$ with $\lambda_0=0$ and $\lvert\lambda\rvert =n$ and a basis element $R\subseteq\undm$ for $T_h(m)$, and consider the element
\begin{equation}\label{eq:element-of-Thm}
\sum_{S\subseteq\undn} (-1)^{\lvert S\rvert} (R\smallsetminus S_\lambda)
\end{equation}
of $T_h(m)$. We need to show that it is always zero when $n>h$, whereas when $n=h$ there exist $\lambda$ and $R$ such that it is non-zero. If $n>h$ there must be some $i\in\{1,\ldots,n\}$ such that $R\cap\{i\}_\lambda = \varnothing$. Then we may write \eqref{eq:element-of-Thm} as the sum over $S\subseteq\undn\smallsetminus\{i\}$ of $(-1)^{\lvert S\rvert} (R\smallsetminus S_\lambda) + (-1)^{\lvert S\rvert +1} ((R\smallsetminus \{i\}_\lambda)\smallsetminus S_\lambda)$, which cancels to zero since $R\smallsetminus\{i\}_\lambda = R$. When $n=h$ we may take $\lambda$ with $\lambda_0=0$ and $\lambda_i=2$ for $i\geq 1$ (so $m=2n$) and $R=\{2,4,\ldots,2n\}$. This completes the proof that $\mathrm{ht}_\oplus(T_h)=h$.

Now we show that $\mathrm{ht}_\cI(T_h)=1$. To begin with, note that to have $\mathrm{ht}_\cI(T_h)\leq 0$ would require that $T_h(r_\varnothing)=\mathrm{id}$, which is not the case, so instead we have $\mathrm{ht}_\cI(T_h) \geq 1$. To see that it is exactly equal to $1$ we need to show that, for all $m\geq n\geq 2$ and any basis element $R$ of $T_h(m)$, the element
\[
\sum_{S\subseteq\{m-n+1,\ldots,m\}} (-1)^{\lvert S\rvert} (R\smallsetminus S)
\]
is zero. The trick is to rewrite this element as the sum over subsets $S\subseteq\{m-n+1,\ldots,m-2\}$ of
\[
(-1)^{\lvert S\rvert} \Bigl( Q + (Q \smallsetminus \{m-1,m\}) - (Q \smallsetminus \{m-1\}) - (Q \smallsetminus \{m\}) \Bigr)
\]
where we have written $Q = R\smallsetminus S$. Since $R$ (and therefore also $Q$) cannot contain both $m-1$ and $m$ (these would be consecutive elements), the four terms above cancel to zero. This completes the proof that $\mathrm{ht}_\cI(T_h)=1$.\hfill $\diamond$

\vspace{1ex}
\noindent (c) \textit{Equality when $T$ admits conjugations.}
To show this we will use the following fact, which is an immediate generalisation of Lemma \ref{lem:isomorphism-of-subobjects}.

\begin{fact}\label{fact:equality-of-two-subobjects}
If $\lambda\vdash m$ is an ordered shifted partition of length $n$ and $T\colon\cC\to\cA$ is a functor, then
\[
\mathrm{im} \biggl( \sum_{S\subseteq\undn} (-1)^{\lvert S\rvert} Ts(f_{\{1,\ldots,\lambda_0\} \cup S_\lambda}) \biggr) \;=\; \mathrm{im}(Ts(f_{\{1,\ldots,\lambda_0\}})) \cap \bigcap_{i=1}^n \mathrm{ker}(Ts(f_{\{i\}_\lambda})).
\]
\end{fact}

Let $T\colon\cC\to\cA$ be a functor and assume that $T$ admits conjugations. Suppose that $\mathrm{ht}_\cI(T)\leq h$. Our aim is to show that $\mathrm{ht}_\oplus(T)\leq h$. In detail, this means the following. Fix $\lambda\vdash m$ with $\lambda_0=0$ and $\lambda_i\geq 1$ for $i\geq 1$ (\cf Remark \ref{rmk:add-assumption-to-defn}) and $\lvert\lambda\rvert =n>h$. In the light of Fact \ref{fact:equality-of-two-subobjects}, our aim is to show that
\begin{equation}\label{eq:compare-two-heights}
\bigcap_{i=1}^n \mathrm{ker}(Ts(f_{\{i\}_\lambda}))
\end{equation}
is zero. Since $\mathrm{ht}_\cI(T)\leq h$, we know (using Lemma \ref{lem:isomorphism-of-subobjects} and the fact that $T$ admits conjugations) that for any $S\subseteq\undm$ of size $\lvert S\rvert >h$,
\[
\mathrm{im}(Ts(f_{\undm\smallsetminus S})) \cap \bigcap_{i\in S} \mathrm{ker}(Ts(f_{\{i\}})) = 0.
\]
We claim that the following equality always holds:
\begin{equation}\label{eq:claim-equality}
\bigcap_{i=1}^n \mathrm{ker}(Ts(f_{\{i\}_\lambda})) \;=\; \bigoplus_{S} \mathrm{im}(Ts(f_{\undm\smallsetminus S})) \cap \bigcap_{i\in S} \mathrm{ker}(Ts(f_{\{i\}})),
\end{equation}
where the direct sum on the right-hand side is taken over all subsets $S\subseteq\undm$ such that for each $i\in\{1,\ldots,n\}$ we have $S\cap\{i\}_\lambda \neq \varnothing$. Any such subset must have size $\lvert S\rvert \geq \lvert \lambda \rvert =n>h$, so -- in our situation -- its contribution to the sum vanishes, and therefore \eqref{eq:compare-two-heights} is zero, as required. So it just remains to prove the equality \eqref{eq:claim-equality}.

\noindent $\bullet\; (\supseteq) :$ Let $S\subseteq\undm$ satisfy the condition above and suppose that $Ts(f_{\{i\}})(x)=0$ for all $i\in S$. For each $j\in\{1,\ldots,n\}$ we may choose $i\in S\cap \{j\}_\lambda$ and compute that
\[
Ts(f_{\{j\}_\lambda})(x) = Ts(f_{\{j\}_\lambda}) \circ Ts(f_{\{i\}})(x) = 0.
\]

\noindent $\bullet\; (\subseteq) :$ Since the idempotents $Ts(f_{\{i\}})$ on $Ts(m)$ pairwise commute there is a decomposition
\begin{align}
Ts(m) \;&=\; \bigoplus_{S\subseteq\undm} \bigcap_{i\in S} \mathrm{ker}(Ts(f_{\{i\}})) \cap \bigcap_{i\in\undm\smallsetminus S} \mathrm{im}(Ts(f_{\{i\}})) \nonumber \\
&=\; \bigoplus_{S\subseteq\undm} \bigcap_{i\in S} \mathrm{ker}(Ts(f_{\{i\}})) \cap \mathrm{im}(Ts(f_{\undm\smallsetminus S})). \label{eq:decomposition}
\end{align}
Now suppose that $x\in Ts(m)$ and $Ts(f_{\{i\}_\lambda})(x)=0$ for each $i\in\{1,\ldots,n\}$. We may write
\[
x=\sum_{S\subseteq\undm} x_S
\]
corresponding to the decomposition \eqref{eq:decomposition}. Note that the endomorphism $Ts(f_{\{i\}_\lambda})$ preserves the decomposition \eqref{eq:decomposition}. Since it is a decomposition as a \emph{direct} sum, we must have $Ts(f_{\{i\}_\lambda})(x_S)=0$ for each $S\subseteq\undm$.

Now, to see that $x$ is contained in the right-hand side of \eqref{eq:claim-equality} we just need to show that if there exists $i\in\{1,\ldots,n\}$ such that $S\cap\{i\}_\lambda = \varnothing$ then $x_S=0$. But we have $x_S\in\mathrm{im}(Ts(f_{\undm\smallsetminus S}))$ and $\{i\}_\lambda \subseteq \undm\smallsetminus S$, so $x_S\in\mathrm{im}(Ts(f_{\{i\}_\lambda}))$. Since $Ts(f_{\{i\}_\lambda})$ is idempotent, this means that
\renewcommand{\qedsymbol}{\ensuremath{\diamond\,\qedsquareb}}
\[
x_S = Ts(f_{\{i\}_\lambda})(x_S) = 0. \qedhere
\]
\end{proof}

\subsection{The injective braid category.}\label{para:full-braids}

Recall from \S\ref{para:degree-WRW} that the \emph{injective braid category} $\cBf(M,X)$ is the subcategory of $\cB(M,X)$ having the same objects (the non-negative integers) and where a morphism of $\cB(M,X)$ -- i.e.\ a (labelled) braid in $M \times [0,1]$ from some subset of $\{a_1,\ldots,a_m\} \times \{0\}$ to some subset of $\{a_1,\ldots,a_n\} \times \{1\}$ -- lies in this subcategory if and only if it has precisely $m$ strands.

The equivalence between the different notions of height discussed in this section suggests how one may extend the notion of height for functors $T \colon \cB(M,X) \to \cA$ to a notion of height for functors $T \colon \cBf(M,X) \to \cA$. If we take our definition of the height of a functor defined on $\cB(M,X)$, which uses the first definition (\S\ref{para:first-def}) of height (see the discussion in \S\ref{para:specialise-this-paper} above), and reinterpret it using instead the second definition (\S\ref{para:second-def}) of height, it may be rewritten in such a way that it involves only morphisms from the subcategory $\cBf(M,X)$. Thus, the height of $T$ depends only on its restriction to $\cBf(M,X)$, and indeed one may use this observation to directly define the height of a functor $T \colon \cBf(M,X) \to \cA$. Explicitly, the definition unravels to the following: $\mathrm{height}(T)\leq h$ if and only if for all $m\geq n>h$ we have
\[
\sum_S \mathrm{coker} \bigl( T(b(\phi_{m,S})) \colon T(\underline{s}) \longrightarrow T(\undm) \bigr) =0,
\]
where the sum is taken over all proper subsets $S$ of $\{ m-n+1, \ldots, m \}$ and $\underline{s}$ denotes $\{ 1,\ldots,\lvert S\rvert \}$. The notation $\phi_{m,S}$ means the unique order-preserving injection $\underline{s} \to \undm$ whose image is equal to $S \subseteq \undm$. In general, given any order-preserving injection $\phi \colon \underline{s} \to \undm$, there is a canonical braid $b(\phi)$ in $M \times [0,1]$ from $\{a_1,\ldots,a_s\} \times \{0\}$ to $\{a_{\phi(1)},\ldots,a_{\phi(s)}\} \times \{1\}$ that realises $\phi$, specified as follows. Recall from \S\ref{sss:some-categories} that the manifold $M$ comes equipped with a collar neighbourhood $c \colon \partial M \times [0,\infty] \hookrightarrow M$ and a basepoint $p \in \partial M$. Let $L$ be the embedded arc $c(\{p\} \times [1,\infty])$ in the interior of $M$. Then $b(\phi)$ is uniquely determined by specifying its endpoints, as above, and that it must be contained in the embedded square $L \times [0,1]$ in $M \times [0,1]$. Labelling each strand of $b(\phi)$ by the constant path at the basepoint $x_0$ of $X$ makes it into a morphism $\underline{s} \to \undm$ of $\cBf(M,X)$.

\subsection{Possible extensions.}\label{para:generalisations}

We finish this section by suggesting potential extensions of the general definitions of \emph{height} given in \S\S\ref{para:first-def}---\ref{para:third-def}. One generalisation, which we have already discussed in detail, is to consider categories $\cC$ equipped with collections of \emph{semi-}functors $\cJ_n \to \cC$, i.e., ``functors'' that are not required to preserve identities (the notation $\cJ_n$ denotes any one of $\I{n}$, $\J{n}$ or $\K{n}$). Another potential generalisation, which was mentioned in \S\ref{para:finite-coproducts}, is to consider twisted coefficient systems (i.e., functors or semi-functors) $T\colon \cC \to \cA$ whose target is a \emph{semi-}abelian category, such as the category $\mathsf{Grp}$ of groups. This is motivated by the work of Hartl, Pirashvili and Vespa~\cite{HartlPirashviliVespa2015Polynomialfunctorsalgebras}, who study functors of the form $\cC \to \mathsf{Grp}$ when $\cC$ admits finite coproducts and a null object.

More fundamentally, one could weaken the structure on $\cC$ by replacing Boolean algebras with posets possessing less structure. If we work in the setting of \S\ref{para:second-def}, then the structure on $\cC$ is given by collections of functors $\J{n} \to \cC$, where $\J{n}$ is the poset of subsets of $\{1,\ldots,n\}$ under inclusion, which is a Boolean algebra. It would be interesting to set up a theory of polynomial functors $\cC \to \cA$ when $\cC$ is instead equipped with collections of functors $L(n) \to \cC$, where the $L(n)$ are lattices with less structure than a Boolean algebra, for example orthocomplemented lattices (in which $\vee$ and $\wedge$ do not necessarily distribute over each other). The lattice of closed subspaces of a Hilbert space is an orthocomplemented lattice, for example, so a natural example to consider would be to take $L(n)$ as the lattice of subspaces of the Hilbert space $\bC^n$.


\section{Partial braid categories}\label{sec:functorial-configuration-spaces}

The paper \cite{Palmer2018Twistedhomologicalstability} is concerned with proving twisted homological stability for the labelled configuration spaces $C_n(M,X)$, for $M$ a connected, open manifold and $X$ a path-connected space. Its twisted coefficient systems are indexed by certain \emph{partial braid categories} $\cB(M,X)$ associated to these data; in that paper they are defined in a slightly ad hoc way, and the \emph{height} and \emph{degree} of a twisted coefficient system on $\cB(M,X)$ is defined in this specific context. In this section, we explain a natural functorial framework into which these constructions fit.

\begin{rmk}
The notions of degree and height used in this section agree with those discussed in the previous two sections (whenever both are defined), but the domains of definition are slightly different. The degree in this section is simply defined as a special case of the degree of \S\ref{sec:inductive-degree}, assuming that the source category is an object of $\cats$ rather than of the larger category $\catst$. The height in this section is defined when the source category is an object of $\cati$. Such an object is in particular a category $\cC$ equipped with a functor $\cI \to \cC$, where $\cI$ is a certain category (\cf \S\ref{para:first-def}). The general definition of height given in \S\ref{sec:cross-effects} specialises to this case, as described in \S\ref{para:specialise-this-paper}, and it agrees with the definition given in this section (see Remark \ref{rmk:two-defs-of-height-agree}).

For this section, we will take the abelian category $\cA$ to be the category $\ab$ of abelian groups. However, this is just in order to preserve notational similarity with \cite{Palmer2018Twistedhomologicalstability}, and in fact everything generalises directly to the setting of an arbitrary abelian category $\cA$.
\end{rmk}

\subsection{Some categories.}\label{sss:some-categories}
We first introduce some $(2,1)$-categories that we will consider. Only the first one has non-identity $2$-morphisms; the other two are really just $1$-categories.

$\bullet\; \mfdc$: Objects are smooth manifolds $M$ (of dimension at least two) equipped with a collar neighbourhood and a basepoint on the boundary. The $1$-morphisms are embeddings preserving collar neighbourhoods and basepoints and $2$-morphisms are isotopies of such embeddings.

More precisely, a collar neighbourhood means a proper embedding
\[
c \colon \partial M \times [0,\infty] \longrightarrow M
\]
such that $c(x,0)=x$ for all $x\in\partial M$. A $1$-morphism from $(M,c_M,p)$ to $(N,c_N,q)$ is then an embedding $f \colon M \hookrightarrow N$ taking $p\in\partial M$ to $q\in\partial N$ and commuting with the collar neighbourhoods, meaning that $f(c_M(x,t)) = c_N(f(x),t)$ for all $x\in\partial M$ and $t\in [0,\infty]$. A $2$-morphism from $f_0$ to $f_1$ is an isotopy $f_s$ such that $f_s(p)=q$ and $f_s(c_M(x,t)) = c_N(f_s(x),t)$ for all $x$, $t$ and $s$.

$\bullet\; \topo$: The category of based, path-connected topological spaces and based continuous maps.

$\bullet\; \cats$: Objects are small $1$-categories $\cC$ equipped with an endofunctor $s\colon \cC \to \cC$ and a natural transformation $\imath \colon \mathrm{id} \to s$. A $1$-morphism from $(\cC,s,\imath)$ to $(\cD,t,\jmath)$ is a functor $f\colon \cC\to \cD$ together with a natural isomorphism $\psi \colon f\circ s \to t\circ f$ of functors $\cC\to \cD$ such that $\jmath * \mathrm{id}_f = \psi \circ (\mathrm{id}_f * \imath)$, where $*$ denotes horizontal composition of natural transformations. (Note that this is a subcategory of the category $\catst$ defined in \S\ref{para:gen-def-degree}.)

\subsection{The partial braid functor.}\label{sss:some-functors}

This is a $2$-functor
\begin{equation}\label{eq:partial-braid-functor}
\cB\colon \mfdc \times \topo \longrightarrow \cats
\end{equation}
such that, for any manifold $M\in\mfdc$ and any space $X\in\topo$, the object $\cB(M,X)\in\cats$ agrees with the category $\cB(M,X)$ defined in \S 2.3 of \cite{Palmer2018Twistedhomologicalstability} together with the extra data defined in \S 3.1 of \cite{Palmer2018Twistedhomologicalstability}.

The definition is as follows. Given objects $(M,c,p) \in \mfdc$ and $(X,x_0) \in \topo$, first set $a_t = c(p,t) \in M$ for $t\in [0,\infty]$ and define an embedding $e \colon M \hookrightarrow M$ by $e(c(m,t)) = c(m,t+1)$ for $(m,t) \in \partial M \times [0,\infty]$ and by the identity outside of the collar neighbourhood. The objects of $\cB(M,X)$ are the non-negative integers. A morphism $m \to n$ is a choice of $k\leq\mathrm{min}(m,n)$ and a path in $C_k(M,X)$, up to endpoint-preserving homotopy, from a subset of $\{(a_1,x_0),\ldots,(a_m,x_0)\}$ to a subset of $\{(a_1,y_0),\ldots,(a_n,y_0)\}$. These may be thought of as braids in $M \times [0,1]$ whose strands have been labelled by loops in $X$ based at $x_0$. Composition is defined by concatenating paths, and then deleting configuration points for which the concatenated path is defined only half-way. For example, omitting the labels, we have the heuristic picture:

\begin{equation}\label{eComposition}
\centering
\begin{split}
\begin{tikzpicture}
[x=1mm,y=1mm]
\node (al1) at (0,0) [fill,circle,inner sep=1pt] {};
\node (al2) at (0,2) [fill,circle,inner sep=1pt] {};
\node (al3) at (0,4) [fill,circle,inner sep=1pt] {};
\node (ar1) at (10,0) [fill,circle,inner sep=1pt] {};
\node (ar2) at (10,2) [fill,circle,inner sep=1pt] {};
\node (ar3) at (10,4) [fill,circle,inner sep=1pt] {};
\node (ar4) at (10,6) [fill,circle,inner sep=1pt] {};
\node (ar5) at (10,8) [fill,circle,inner sep=1pt] {};
\draw (al1) .. controls (5,0) and (5,4) .. (ar3);
\draw[white,line width=1mm] (al3) .. controls (5,4) and (5,0) .. (ar1);
\draw (al3) .. controls (5,4) and (5,0) .. (ar1);
\draw[white,line width=1mm] (al2) .. controls (5,2) and (5,8) .. (ar5);
\draw (al2) .. controls (5,2) and (5,8) .. (ar5);
\node at (13,4) {$\circ$};
\begin{scope}[xshift=16mm]
\node (bl1) at (0,0) [fill,circle,inner sep=1pt] {};
\node (bl2) at (0,2) [fill,circle,inner sep=1pt] {};
\node (bl3) at (0,4) [fill,circle,inner sep=1pt] {};
\node (bl4) at (0,6) [fill,circle,inner sep=1pt] {};
\node (bl5) at (0,8) [fill,circle,inner sep=1pt] {};
\node (br1) at (10,0) [fill,circle,inner sep=1pt] {};
\node (br2) at (10,2) [fill,circle,inner sep=1pt] {};
\node (br3) at (10,4) [fill,circle,inner sep=1pt] {};
\node (br4) at (10,6) [fill,circle,inner sep=1pt] {};
\draw (bl1) .. controls (5,0) and (5,2) .. (br2);
\draw (bl2) .. controls (5,2) and (5,6) .. (br4);
\draw[white,line width=1mm] (bl5) .. controls (5,8) and (5,4) .. (br3);
\draw (bl5) .. controls (5,8) and (5,4) .. (br3);
\end{scope}
\node at (32,2.5) {$=$};
\begin{scope}[xshift=38mm]
\node (cl1) at (0,0) [fill,circle,inner sep=1pt] {};
\node (cl2) at (0,2) [fill,circle,inner sep=1pt] {};
\node (cl3) at (0,4) [fill,circle,inner sep=1pt] {};
\node (cr1) at (10,0) [fill,circle,inner sep=1pt] {};
\node (cr2) at (10,2) [fill,circle,inner sep=1pt] {};
\node (cr3) at (10,4) [fill,circle,inner sep=1pt] {};
\node (cr4) at (10,6) [fill,circle,inner sep=1pt] {};
\draw (cl3) .. controls (5,4) and (5,2) .. (cr2);
\draw[white,line width=1mm] (cl2) .. controls (5,2) and (5,4) .. (cr3);
\draw (cl2) .. controls (5,2) and (5,4) .. (cr3);
\end{scope}
\end{tikzpicture}
\end{split}
\end{equation}

The endofunctor $s \colon \cB(M,X) \to \cB(M,X)$ sends the object $n$ to $n+1$ and sends a morphism $\gamma$, which is a path in $C_k(M,X)$, to the morphism $s_k \circ \gamma$, where $s_k \colon C_k(M,X) \to C_{k+1}(M,X)$ sends a configuration $\{(m_1,x_1),\ldots,(m_k,x_k)\}$ to $\{(a_1,x_0),(e(m_1),x_1),\ldots,(e(m_k),x_k)\}$.

The natural transformation $\imath \colon \mathrm{id} \to s$ consists of the morphisms $n \to n+1$ given by the paths $t \mapsto \{(a_{1+t},x_0),\ldots,(a_{n+t},x_0)\}$.

Given $1$-morphisms $\phi \colon (M,c_M,p) \to (N,c_N,q)$ and $f\colon (X,x_0) \to (Y,y_0)$, we need to specify a functor $F \colon \cB(M,X) \to \cB(N,Y)$ and a natural isomorphism $\psi \colon F \circ s \to s \circ F$. In fact, we will define $F$ such that $F \circ s = s \circ F$ and take $\psi$ to be the identity. On objects, we define $F$ to be the identity. A morphism $\gamma$ in $\cB(M,X)$ -- represented by a path in $C_k(M,X)$ for some $k$ -- is sent by $F$ to the morphism in $\cB(N,Y)$ represented by the path $C_k(\phi,f) \circ \gamma$, where $C_k(\phi,f) \colon C_k(M,X) \to C_k(N,Y)$ sends a configuration $\{(m_1,x_1),\ldots,(m_k,x_k)\}$ to $\{(\phi(m_1),f(x_1)),\ldots,(\phi(m_k),f(x_k))\}$.

If $\phi^\prime$ is another $1$-morphism (i.e., embedding) that is connected to $\phi$ by a $2$-morphism (i.e., is isotopic to $\phi$ respecting basepoints and collar neighbourhoods), then applying the above construction to $\phi^\prime$ and $f$, instead of $\phi$ and $f$, results in exactly the same functor $F \colon \cB(M,X) \to \cB(N,Y)$. Thus $\cB$ extends to a $2$-functor by sending all $2$-morphisms to identities.

\subsection{Degree.}\label{sss:degree}

Definition \ref{def:degree-general} from \S\ref{sec:inductive-degree} specialises to associate a \emph{degree}
\[
\degree(T)\in\{-1,0,1,2,3,\ldots,\infty\}
\]
to any functor $T\colon \cC\to\mathsf{Ab}$ for any object $\cC\in\cats$. In particular, via the functor \eqref{eq:partial-braid-functor} above, it associates a \emph{degree} to any functor $\cB(M,X) \to \ab$, and recovers the notion of \emph{degree} used in \cite{Palmer2018Twistedhomologicalstability}.

\begin{lem}\label{l:degree-under-composition}
If $f\colon \cC\to \cD$ is a morphism in $\cats$ and $T\colon \cD\to \mathsf{Ab}$ is any functor, then we have the inequality $\mathrm{deg}(Tf) \leq \mathrm{deg}(T)$. If $f$ is essentially surjective on objects then it is an equality.
\end{lem}

\begin{proof}
We need to show that $\mathrm{deg}(T)\leq n \Rightarrow \mathrm{deg}(Tf)\leq n$ for each $n\geq -1$, and that the reverse implication also holds if $f$ is essentially surjective on objects. The base case $n=-1$ is clear, since $\mathrm{deg}(T)=-1$ simply means that $T=0$. It is then an exercise in
elementary $2$-category theory to show that $(\Delta T)f \cong \Delta (Tf)$, from which the inductive step follows.
\end{proof}

\begin{defn}\label{def:braidable}
Say that a category $(\cC,s,\imath) \in \cats$ is \emph{braidable} if there exists a natural isomorphism $\Psi \colon s \circ s \to s \circ s$ such that $\imath * \mathrm{id}_s = \Psi \circ (\mathrm{id}_s * \imath)$. Note that this is equivalent to saying that the endofunctor $s \colon \cC \to \cC$ itself may be extended to a morphism of $\cats$.
\end{defn}

\begin{coro}\label{coro:braidable}
If $(\cC,s,\imath) \in \cats$ is braidable, and $T \colon \cC \to \mathsf{Ab}$ is any functor, then we have the inequality $\mathrm{deg}(Ts) \leq \mathrm{deg}(T)$, which is an equality if $s$ is essentially surjective on objects.
\end{coro}

\subsection{Uniformly-defined twisted coefficient systems.}\label{sss:uniform-coeff-systems}

Given $\cC\in\cats$, a twisted coefficient system is simply a functor $\cC\to\ab$. More generally, we may start with a diagram $F\colon\cD\to\cats$ and define a twisted coefficient system for each object of $\cD$ in a compatible way, as follows. By abuse of notation, write $F$ also for the composition $\cD\to\cats\to\mathsf{Cat}$ of $F$ with the forgetful functor down to the category of small categories. A \emph{uniformly-defined twisted coefficient system for $F$} is then a functor
\[
T\colon \cD {\textstyle\int} F \longrightarrow \ab
\]
with domain the Grothendieck construction of $F$. For each object $d\in\cD$ there is a natural functor $j_d\colon F(d) \to \cD {\textstyle\int} F$, so this determines a twisted coefficient system $T_d \colon F(d) \to \ab$ for each $d\in\cD$.

\begin{lem}\label{l:grothendieck-construction}
The category $\cD {\textstyle\int} F$ is naturally an object of $\cats$ and $j_d$ is a morphism of $\cats$.
\end{lem}

Thus we have a well-defined degree $\mathrm{deg}(T)$ of a uniformly-defined twisted coefficient system $T$, and by Lemma \ref{l:degree-under-composition} we know that $\mathrm{deg}(T) \geq \mathrm{deg}(T_d)$, in other words it is an upper bound on the degrees of the individual twisted coefficient systems $T_d \colon F(d)\to\cD{\textstyle\int}F\to\ab$.

\begin{proof}[Proof of Lemma \ref{l:grothendieck-construction}]
For each object $d\in\cD$ the category $F(d)$ is equipped with an endofunctor, which we will denote $s_d\colon F(d)\to F(d)$, and a natural transformation $\iota_d \colon 1_{F(d)} \Rightarrow s_d$. Recall that $\cD{\textstyle\int}F$ has objects $(d,x)$ for $d\in\cD$ and $x\in F(d)$ and morphisms $(f,g) \colon (d,x)\to (d^\prime,x^\prime)$ where $f\colon d\to d^\prime$ in $\cD$ and $g\colon F(f)(x)\to x^\prime$ in $F(d^\prime)$. One may then define an endofunctor $\bar{s}$ on $\cD{\textstyle\int}F$ by setting $\bar{s}(d,x) = (d,s_d(x))$ and $\bar{s}(f,g) = (f,s_{d^\prime}(g))$, and a natural transformation $\bar{\iota} \colon 1_{\cD{\scriptstyle\int}F} \Rightarrow \bar{s}$ by setting $\bar{\iota}_{(d,x)} = (1_d,(\iota_d)_x)$.

This makes $\cD{\textstyle\int}F$ into an object of $\cats$ and the functor $j_d \colon F(d) \to \cD {\textstyle\int} F$, together with $\psi = \mathrm{id} \colon j_d \circ s_d \to \bar{s} \circ j_d$, into a morphism of $\cats$.
\end{proof}

For example, we could take $\cD$ to be $\mfdc \times \topo$ and $F$ to be the functor \eqref{eq:partial-braid-functor}, in which case a ``uniformly-defined twisted coefficient system'' determines twisted coefficient systems for all \emph{partial braid categories} $\cB(M,X)$ simultaneously.

\begin{rmk}\label{r:cocone}
Fix $F\colon \cD\to\cats$ and suppose we are given a cocone on $F$ (i.e.\ an object $\cC\in\cats$ and a natural transformation $\alpha\colon F\Rightarrow \mathrm{const}_\cC$) together with a functor $T\colon \cC\to\ab$. This determines a functor $\mathbb{T}\colon\cD{\textstyle\int}F\to\ab$ given on objects by $\mathbb{T}(d,x)=T(\alpha_d(x))$. One can show inductively that in this setting we have an inequality $\mathrm{deg}(\mathbb{T})\leq\mathrm{deg}(T)$. 

Note that the category $\Sigma$ of finite cardinals and partially-defined injections is naturally an object of $\cats$ if one equips it with the endofunctor taking $n$ to $n+1$ and a morphism $f$ to the morphism defined by $1 \mapsto 1$ and $i \mapsto f(i-1)+1$ for $i\geq 2$, together with the natural transformation given by the collection of morphisms $\iota_n\colon n \to n+1$ defined by $\iota_n(i)=i+1$.
We may therefore consider the slice category $(\cats \downarrow \Sigma)$. A lift of a functor $F \colon \cD \to \cats$ to the slice category is the same thing as a cocone on $F$ with $\Sigma\in\cats$ at its ``tip''. So if we fix a functor $F \colon \cD \to (\cats \downarrow \Sigma)$, any twisted coefficient system on $\Sigma$ (i.e.\ functor $\Sigma\to\ab$) automatically induces a uniformly-defined twisted coefficient system (i.e.\ functor $\cD{\textstyle\int}F\to\ab$) of the same or smaller degree.

In particular, the functor $\cB$ \eqref{eq:partial-braid-functor} naturally lifts to the slice category\footnote{\label{f:Whitney}One may see this claim as follows. The category $\mfdc \times \topo$ has a \emph{cofinal} subcategory consisting of (collared, basepointed) Euclidean halfspaces of dimension $\geq 3$ in $\mfdc$, together with the one-point space $* \in \topo$. Cofinality of this subcategory follows from the Whitney Embedding Theorem, or, more precisely, its analogue for manifolds with collared boundary (see Lemma \ref{lem:Whitney-with-boundary}). The functor $\cB$ sends this whole subcategory to the object $\Sigma$ (and its identity morphism) in $\cats$ (\cf \S 2.4 of \cite{Palmer2018Twistedhomologicalstability}), thus automatically providing a lift of $\cB$ to $(\cats \downarrow \Sigma)$.} (\cf the construction of \eqref{eq:partial-braid-functor-v2} below), so a twisted coefficient system for $\Sigma$ induces a uniformly-defined twisted coefficient system for $\cB$, and thence twisted coefficient systems for each $\cB(M,X)$.
\end{rmk}

\subsection{Height.}\label{sss:height}

The definition of the \emph{height} of a twisted coefficient system $T \colon \cC \to \ab$ requires a different structure on the source category $\cC$.

Recall from \S\ref{para:compare-two-heights} that $\cI$ and $\Sigma$ have objects $0,1,2,\ldots$, morphisms $m \to n$ of $\Sigma$ are partially-defined injections $\undm \to \undn$ and morphisms of $\cI$ are those partially-defined injections that are the identity wherever they are defined. Their automorphism groups are the symmetric groups $\Sigma_n$ and trivial respectively. Denote their endomorphism monoids by $\cP_n = \mathrm{End}_\Sigma(n)$ and $\cI_n = \mathrm{End}_\cI(n)$. Note that $\cI_n$ is the submonoid of $\cP_n$ of all idempotent elements. It may also be described as the monoid of subsets of $\{1,\ldots,n\}$ under the operation $\cap$ with neutral element $\{1,\ldots,n\}$, or under the operation $\cup$ with neutral element $\varnothing$. The latter identification is given by associating to a subset $S\subseteq\{1,\ldots,n\}$ the morphism $f_{n,S} \in \cI_n$ that ``forgets'' $S$, in other words the partial injection from $\{1,\ldots,n\}$ to itself that is undefined on $S$ and the identity on $\{1,\ldots,n\} \smallsetminus S$.

\begin{defn}\label{def:cati}
Let $\cati$ be the category whose objects are small categories $\cC$ equipped with functors $s\colon \cI \to \cC$ and $\pi \colon \cC \to \Sigma$ such that $\pi \circ s$ is the inclusion, and such that the following two conditions are satisfied:
\begin{itemizeb}
\item The homomorphisms $\pi \colon \mathrm{End}_\cC(s(n)) \to \cP_n$ and $\pi \colon \mathrm{Aut}_\cC(s(n)) \to \Sigma_n$ are surjective.
\item (``Locality'') Fix $n\geq 0$ and $\phi\in\mathrm{End}_\cC(s(n))$. For each $i$ there exists $j$ and for each $j$ there exists $i$ such that
\begin{equation}\label{eq:locality}
\phi \circ s(f_{n,\{i\}}) \;=\; s(f_{n,\{j\}}) \circ \phi.
\end{equation}
\end{itemizeb}
Morphisms from $(\cC,s,\pi)$ to $(\cC^\prime,s^\prime,\pi^\prime)$ are functors $f \colon \cC \to \cC^\prime$ such that $f \circ s = s^\prime$ and $\pi = \pi^\prime \circ f$.
\end{defn}

There is an analogue of the functor \eqref{eq:partial-braid-functor} for this setting, which we denote by the same letter,
\begin{equation}\label{eq:partial-braid-functor-v2}
\cB\colon \mfdc \times \topo \longrightarrow \cati ,
\end{equation}
and which is defined as follows. Given objects $M \in \mfdc$ and $X \in \topo$, the category $\cB(M,X)$ itself is defined as in \S\ref{sss:some-functors}. Now we define functors $s \colon \cI \to \cB(M,X)$ and $\pi \colon \cB(M,X) \to \Sigma$. On objects, $s$ is the identity. If $f \colon m \to n$ is the morphism in $\cI$ that is the identity on $S \subseteq \{1,\ldots,\mathrm{min}(m,n)\}$ and undefined elsewhere, define $s(f)$ to be the (homotopy class of the) constant path in $C_{\lvert S \rvert}(M,X)$ at the point $\{ (a_s,x_0) \mid s \in S \}$. The functor $\pi$ is also the identity on objects. A morphism $m \to n$ in $\cB(M,X)$ is determined by a path of configurations from some subconfiguration of $\{a_1,\ldots,a_m\}$ to some subconfiguration of $\{a_1,\ldots,a_n\}$ (together with some labels, which we are ignoring). This induces a partial injection from $\undm$ to $\undn$, and the functor $\pi$ records precisely this information.

The locality property \eqref{eq:locality} is satisfied since deleting the $i$th strand of a braid from one end corresponds to deleting the $j$th strand from the other end for some $j$. If there is no $i$th strand, according to the ordering at one end, then we may take $j$ to be a number such that there is no $j$th strand at the other end, so that both sides of \eqref{eq:locality} are equal to $\phi$. The surjectivity property holds since any (partial) injection may be realised by a (partial) braid on $M$, since manifolds $M \in \mfdc$ are required to have dimension at least two.

\paragraph*{An alternative viewpoint.}

Since a category $\cC \in \cati$ in particular comes equipped with a functor $s \colon \cI \to \cC$, we have from \S\ref{sec:cross-effects} a definition of the \emph{height}
\[
\mathrm{height}(T)\in\{-1,0,1,2,3,\ldots,\infty\}
\]
of any functor $T \colon \cC \to \ab$, as described in \S\ref{para:specialise-this-paper}. In particular, via the functor \eqref{eq:partial-braid-functor-v2}, it associates a \emph{height} to any functor $\cB(M,X) \to \ab$, and recovers the notion of \emph{height} used in \cite{Palmer2018Twistedhomologicalstability}.

In the next section we describe the definition of the \emph{height} from a different viewpoint, which depends on the full structure of $\cC$ as an object of $\cati$, not just on the functor $s \colon \cI \to \cC$.
This may be summarised as follows. An object $\cC \in \cati$ has associated categories and faithful functors $\cA \to \cB \subseteq \cC$, together with an $\bN$-grading of the objects of $\cA$. We may therefore filter the category $\cA$ by defining $\cA^{>n} \subseteq \cA$ to be the full subcategory on objects with grading more than $n$, for $n \in \{-1,0,1,2,\ldots,\infty\}$. Now given any functor $T \colon \cC \to \ab$, there is an associated functor $T^\prime \colon \cA \to \ab$ related to $T$ by the fact that the induced functor $\mathrm{Ind}_{\cA \to \cB}(T^\prime)$ is isomorphic to $T$ on the subgroupoid $\cB^{\sim}$ (the underlying groupoid of $\cB$). The \emph{height} of $T$ is then
\[
\mathrm{height}(T) = \mathrm{min}\bigl\lbrace n \bigm| T^\prime \equiv 0 \, \text{ on } \cA^{>n} \bigr\rbrace .
\]

The idea is that the functor $T^\prime \colon \cA \to \ab$ records all of the \emph{cross-effects} of $T \colon \cC \to \ab$ simultaneously, with the grading indicating which cross-effects correspond to which objects of $\cC$. This viewpoint could perhaps be useful in generalising the notion of \emph{height} to more sophisticated situations, by allowing the structure of the category $\cA$ indexing the cross-effects to be more complicated (here it is just a disjoint union of monoids).

The details of this alternative viewpoint on the \emph{height} of a functor are given in \S\ref{sss:height-general}, using some facts about induction for representations of categories, which may be of interest in their own right, which we discuss in \S\S\ref{sss:induction}--\ref{sss:special-cases}. We explain in Remark \ref{rmk:two-defs-of-height-agree} why this alternative viewpoint agrees exactly with the definition from \S\ref{para:specialise-this-paper} above.


\section{Induction for representations of categories}\label{sec:representations-of-categories}

In this section we give details of the alternative viewpoint on the \emph{height} of a twisted coefficient system $T \colon \cC \to \cA$, when $\cC$ is an object of the category $\cati$ defined in \S\ref{sss:height} immediately above. We begin with a detour through the notion of induction for representations of categories in \S\S\ref{sss:induction}--\ref{sss:special-cases}, and then return to the alternative definition of the \emph{height} of a twisted coefficient system in \S\ref{sss:height-general}.

\subsection{Induction for representations of categories.}\label{sss:induction}

We will take $\bZ$ as our ground ring in this section, but everything works equally well over an arbitrary commutative unital ring. Suppose that we have a functor $f\colon \cA\to \cB$ and we wish to extend representations of $\cA$, i.e., functors $g\colon \cA\to\ab = \bZ\text{-mod}$, along $f$ to $\cB$. We will define a functor (``induction along $f$'')
\begin{equation}\label{eq:induction-functor}
\mathrm{Ind}_f \colon \mathsf{Fun}(\cA,\ab) \longrightarrow \mathsf{Fun}(\cB,\ab)
\end{equation}
that does this, and prove a few of its properties. We note that what we will be defining is simply the \emph{left Kan extension} operation along the functor $f$, but we would like to have explicit formulas for this, so we will give an elementary definition instead of using this universal characterisation.

First we explain the notion of a category ring. Given any category $\cA$, its \emph{category ring} $\bZ \cA$ is defined as follows: as an abelian group it is freely generated by the morphisms of $\cA$, and the product of two morphisms is their composition if they are composable and zero otherwise.\footnote{More generally, there is a ring associated to any \emph{semigroup with absorbing element}, i.e.\ semigroup $S$ containing an element $\infty$ such that $s\infty = \infty s = \infty$ for all $s\in S$. This ring is $\bZ S/\bZ\{\infty\}$: the free ring without unit $\bZ S$ generated by $S$ quotiented by the two-sided ideal $\bZ\{\infty\}$ generated by $\infty$. A category $\cC$ may be regarded as a \emph{partial semigroup} and then turned into a semigroup with absorbing element $\cC^\circ$ by adjoining a new element $\infty$: any composition $fg$ that is undefined in $\cC$ is defined to be $\infty$ in $\cC^\circ$. This recovers the definition of \emph{category ring} given above. The construction is similar to that of \cite{Boettger2016Monoidswithabsorbing}, which associates a ring to any \emph{partial monoid}, going via a \emph{monoid with absorbing element}, called a \emph{binoid} in the cited paper.} Note that $\bZ\cA$ is unital if and only if $\cA$ has finitely many objects, in which case the unit is given by the formal sum of the identities $1_a$ over all objects $a$ of $A$. This definition was given by B.\ Mitchell in \S 7 of \cite{Mitchell1972Ringswithseveralobjects}, see also \S 2 of \cite{Webb2007introductionrepresentationscohomology}. (We note that the definition of Mitchell is more general: it associates a ring $[\cC]$ to each \emph{preadditive} ($\ab$-enriched) category $\cC$; the above definition of $\bZ \cA$ is recovered as $[\cA_{\ab}]$, where $\cA_{\ab}$ denotes the free preadditive category generated by $\cA$.) Now, to a functor $f\colon \cA\to \cB$ and an object $b$ of $\cB$ we may associate the following right $\bZ \cA$-module:
\[
\bZ (f,b) = \bZ \Bigl\langle (\beta,a) \bigm| a\in\mathrm{ob}(\cA),\; \beta\colon f(a)\to b \text{ in } \cB \Bigr\rangle
\]
with the $\bZ \cA$ action defined as follows: a morphism $\alpha\colon a_1 \to a_2$ sends $(\beta,a)$ to zero if $a\neq a_2$ and to $(\beta\circ f(\alpha),a_1)$ otherwise. (This could be written more compactly in terms of ``heteromorphisms'' \cite{Ellerman2007Adjointfunctorsheteromorphisms} as $\bigoplus_a \bZ\mathrm{Het}_f(a,b)$, but this will not be relevant for us here.)

Given a representation $g\colon \cA\to\ab$ we may define a left $\bZ \cA$-module:
\[
g(\mathrm{ob} \cA) = \bigoplus_{a\in\mathrm{ob}(\cA)}g(a),
\]
with $\alpha\colon a_1\to a_2$ sending $x\in g(a)$ to zero if $a\neq a_1$ and to $g(\alpha)(x)\in g(a_2)$ otherwise. We now define $\mathrm{Ind}_f(g)\colon \cB\to\ab$ as follows:
\begin{equation}\label{eq:induction-fg}
\begin{aligned}
&\text{on objects:}\qquad& &\phantom{=}\mathrm{Ind}_f(g)(b) \;=\; \bZ (f,b) \,\otimes_{\bZ \cA}\, g(\mathrm{ob} \cA) \\
&\text{on morphisms:}\qquad& &\phantom{=}\mathrm{Ind}_f(g)(\gamma\colon b\to b^\prime) \colon (\beta,a)\otimes x \;\mapsto\; (\gamma\circ\beta,a) \otimes x.
\end{aligned}
\end{equation}
We note that $\mathrm{Ind}_f(g)(b)$ is generated by elements of the form $(\beta,a)\otimes x$ with $x\in g(a)$.\footnote{This is because it is clearly generated by elements of this form with $x\in g(a^\prime)$ for $a^\prime$ possibly different to $a$, but if $a\neq a^\prime$ this element is in fact zero, since then $(\beta,a)\otimes x = (\beta,a)\otimes g(\mathrm{id}_{a^\prime})(x) = (\beta,a)\cdot \mathrm{id}_{a^\prime} \otimes x = 0\otimes x = 0$.} This defines the functor $\mathrm{Ind}_f$ on objects, i.e.\ representations of $\cA$. Given a natural transformation $\tau\colon g\Rightarrow g^\prime$ we define the natural transformation $\mathrm{Ind}_f(\tau) \colon \mathrm{Ind}_f(g) \Rightarrow \mathrm{Ind}_f(g^\prime)$ by
\begin{flalign*}
&&& \mathrm{Ind}_f(\tau)_b \colon (\beta,a)\otimes x \;\mapsto\; (\beta,a)\otimes \tau_{a}(x). &&(\text{where } x\in g(a))
\end{flalign*}
This completes the definition of the induction functor \eqref{eq:induction-functor}.

As mentioned above, one can check that this explicit construction is left adjoint to the restriction functor $(-) \circ f$; in other words, it is the left Kan extension operation: $\mathrm{Ind}_f = \mathrm{Lan}_f$.

\subsection{Comparison to induction for modules over category rings.}\label{ss:induction}

The construction mentioned above, taking a representation $g\colon \cA\to\ab$ to the $\bZ \cA$-module $g(\mathrm{ob}\cA)$, in fact defines an embedding
\[
\mathsf{Fun}(\cA,\ab) \longrightarrow \bZ \cA\text{-mod}
\]
of the representation category of $\cA$ as a full subcategory of the category of left $\bZ \cA$-modules. The image is the full subcategory on those $\bZ \cA$-modules $M$ such that for each element $m\in M$ the set $\{a\in\mathrm{ob}(\cA) \mid 1_a \cdot m \neq 0 \}$ is finite. Hence if $\cA$ has only finitely many objects, this is an equivalence of categories. This is Theorem 7.1 of \cite{Mitchell1972Ringswithseveralobjects}; see also Proposition 2.1 of \cite{Webb2007introductionrepresentationscohomology}.

A functor $f\colon \cA\to \cB$ induces a homomorphism of abelian groups $\bZ f\colon \bZ \cA \to \bZ \cB$ that is a homomorphism of (non-unital) \emph{rings} if and only if $f$ is \emph{injective on objects} (see Proposition 2.2.3 of \cite{Xu2007Representationscategoriesapplications} and Proposition 3.1 of \cite{Webb2007introductionrepresentationscohomology}). In this case $\bZ \cB$ may be considered as a right module over $\bZ \cA$ via the ring homomorphism $\bZ f$ and hence there is an induction functor
\begin{equation}\label{eq:induction-functor-2}
\bZ \cB \otimes_{\bZ \cA} - \colon \bZ \cA\text{-mod} \longrightarrow \bZ \cB\text{-mod}.
\end{equation}
This agrees with our definition above:

\begin{lem}
When $f$ is injective on objects so that the right-hand vertical arrow below exists, the following square commutes up to natural isomorphism\textup{:}
\begin{equation}\label{eq:comparing-induction-functors}
\centering
\begin{split}
\begin{tikzpicture}
[x=1mm,y=1mm]
\node (tl) at (0,15) {$\mathsf{Fun}(\cA,\ab)$};
\node (tr) at (40,15) {$\bZ \cA\text{\textup{-mod}}$};
\node (bl) at (0,0) {$\mathsf{Fun}(\cB,\ab)$};
\node (br) at (40,0) {$\bZ \cB\text{\textup{-mod}}$};
\draw[->] (tl) to node[left,font=\small]{$\mathrm{Ind}_f$} (bl);
\draw[->] (tr) to node[right,font=\small]{$\bZ \cB\otimes_{\bZ \cA}-$} (br);
\incl{(tl)}{(tr)}
\incl{(bl)}{(br)}
\end{tikzpicture}
\end{split}
\end{equation}
\end{lem}
\begin{proof}
As a right $\bZ \cB$-module, $\bZ \cB$ itself is isomorphic to the direct sum $\bigoplus_b \bZ \mathrm{Hom}_\cB(\cB,b)$ where the sum is over all objects $b$ of $\cB$ and the notation $\mathrm{Hom}_\cB(\cB,b)$ denotes the disjoint union of the sets $\mathrm{Hom}_\cB(b^\prime,b)$ over all objects $b^\prime$ of $\cB$. This may be viewed as an isomorphism of right $\bZ \cA$-modules via $\bZ f$. Moreover, under the hypothesis that $f$ is injective on objects, the right $\bZ \cA$-module $\bZ(f,b)$ is isomorphic to $\bZ\mathrm{Hom}_\cB(\cB,b)$. Hence we have isomorphisms of right $\bZ \cA$-modules
\[
\bZ \cB \;\cong\; \textstyle{\bigoplus_b}\, \bZ (f,b).
\]
Now the result of sending a functor $g\colon \cA\to\ab$ clockwise around the diagram is $\bZ \cB \otimes_{\bZ \cA} g(\mathrm{ob}\cA)$ whereas the result of sending it anticlockwise around the diagram is $\bigoplus_b \bZ(f,b) \otimes_{\bZ \cA} g(\mathrm{ob}\cA)$.
\end{proof}

It does not seem clear how to extend $\mathrm{Ind}_f$ to an induction functor $\bZ \cA\text{-mod} \to \bZ \cB\text{-mod}$ in the case when $f$ is not injective on objects.

\begin{rmk}
Under certain conditions (although certainly not in general) induction followed by restriction is isomorphic to the identity. More precisely, write $\mathrm{Res}_f(-) = (-)\circ f \colon \mathsf{Fun}(\cB,\ab) \to \mathsf{Fun}(\cA,\ab)$, so that $\mathrm{Res}_f \circ \mathrm{Ind}_f$ is an endofunctor of $\mathsf{Fun}(\cA,\ab)$. Then there is a natural transformation $\mathrm{id} \Rightarrow \mathrm{Res}_f \circ \mathrm{Ind}_f$ (the unit of the adjunction $\mathrm{Ind}_f \dashv \mathrm{Res}_f$) with the property that for each $g\in\mathsf{Fun}(\cA,\ab)$ and $a\in \cA$ its component $g(a) \to \mathrm{Ind}_f(g)(f(a))$ is surjective if $f$ is full and bijective if $f$ is also faithful. So when $f$ is fully faithful the composition $\mathrm{Res}_f \circ \mathrm{Ind}_f$ is isomorphic to the identity. We leave this assertion without proof since we will not use it (but see \S 3 of \cite{Webb2007introductionrepresentationscohomology}, especially Prop.\ 3.2(1), for further discussion).
\end{rmk}

\subsection{Special cases.}\label{sss:special-cases}

We note that the formula \eqref{eq:induction-fg} for the induced functor simplifies in some special cases. Suppose first that $\cA$ is a disjoint union of monoids, i.e., has no morphisms between distinct objects. Then $\bZ \cA$ splits as a direct sum of rings $\bigoplus_a \bZ\mathrm{End}_\cA(a)$. Also, the right $\bZ \cA$-module $\bZ(f,b)$ splits as a direct sum of modules $\bigoplus_a \bZ\mathrm{Hom}_\cB(f(a),b)$ and the left $\bZ \cA$-module splits as a direct sum of modules $\bigoplus_a g(a)$. The tensor product therefore splits in the same way, and we have:
\[
\mathrm{Ind}_f(g)(b) \;\cong\; \textstyle{\bigoplus}_a \bigl( \bZ\mathrm{Hom}_\cB(f(a),b) \otimes_{\bZ\mathrm{End}_\cA(a)} g(a) \bigr).
\]
If the category $\cB$ is also a disjoint union of monoids, then this simplifies further to
\[
\mathrm{Ind}_f(g)(b) \;\cong\; \textstyle{\bigoplus}_{a\in f^{-1}(b)} \bigl( \bZ\mathrm{End}_\cB(b) \otimes_{\bZ\mathrm{End}_\cA(a)} g(a) \bigr).
\]

Under certain conditions, this may be written purely in terms of automorphism groups, rather than endomorphism monoids, using the following elementary lemma.

\begin{lem}\label{l:induction-and-restriction-commute}
Suppose that the square of submonoids
\begin{center}
\begin{tikzpicture}
[x=1mm,y=1mm]
\node (tl) at (0,10) {$C$};
\node (tr) at (20,10) {$D$};
\node (bl) at (0,0) {$A$};
\node (br) at (20,0) {$B$};
\incl{(tl)}{(tr)}
\incl{(bl)}{(br)}
\incl{(bl)}{(tl)}
\incl{(br)}{(tr)}
\end{tikzpicture}
\end{center}
satisfies the following condition $(*)$\textup{:} there is a subset $X\subseteq B\times C$ such that the multiplication map $X\to D$ is surjective and whenever $b_1c_1=b_2c_2$ for $(b_i,c_i)\in X$ there exists $a\in A$ such that $b_1=b_2a$ and $ac_1=c_2$. Then for any $\bZ C$-module $M$ there is an isomorphism of $\bZ B$-modules\textup{:}
\[
\bZ D\otimes_{\bZ C} M \;\cong\; \bZ B\otimes_{\bZ A} M.
\]
\end{lem}

One may also write this as $\mathrm{Res}_B^D (\mathrm{Ind}_C^D (M)) \cong \mathrm{Ind}_A^B (\mathrm{Res}_A^C (M))$.

\begin{proof}
There is an obvious $\bZ B$-module homomorphism $i\colon \bZ B\otimes_{\bZ A} M \to \bZ D\otimes_{\bZ C} M$ given by $b\otimes m \mapsto b\otimes m$. To define an inverse, note that by property $(*)$ there is a well-defined function $D\times M \to \bZ B\otimes_{\bZ A} M$ given by sending $(d,m)$ to $b\otimes c\cdot m$, where $(b,c)\in X$ such that $bc=d$. This is linear in the second entry, and it sends $(dc,m)$ and $(d,c\cdot m)$ to the same element for any $c\in C$, so it induces a homomorphism $\bZ D\otimes_{\bZ C} M \to \bZ B\otimes_{\bZ A} M$. This is an inverse for $i$.
\end{proof}

The condition $(*)$ in Lemma \ref{l:induction-and-restriction-commute} will be valid in our setting by the following lemma. Let $\cP_n$ be the monoid of partial bijections of $\{1,\ldots,n\}$ and write $\cP_k \times \cP_{n-k}$ for its submonoid of those partial bijections that preserve the partition into $\{1,\ldots,n-k\}$ and $\{n-k+1,\ldots,n\}$. Write $D^\sim$ for the underlying group of a monoid $D$, so for example $(\cP_n)^\sim = \Sigma_n$ is the $n$th symmetric group.

\begin{lem}\label{l:checking-property-star}
Suppose that $\pi\colon D\to \cP_n$ is a surjective monoid homomorphism such that the homomorphism of underlying groups $D^\sim \to \Sigma_n$ is also surjective. Define $C=\pi^{-1}(\cP_k \times \cP_{n-k})$. Then the square of submonoids
\begin{equation}\label{eq:square-of-submonoids}
\centering
\begin{split}
\begin{tikzpicture}
[x=1mm,y=1mm]
\node (tl) at (0,10) {$C$};
\node (tr) at (20,10) {$D$};
\node (bl) at (0,0) {$C^\sim$};
\node (br) at (20,0) {$D^\sim$};
\incl{(tl)}{(tr)}
\incl{(bl)}{(br)}
\incl{(bl)}{(tl)}
\incl{(br)}{(tr)}
\end{tikzpicture}
\end{split}
\end{equation}
satisfies condition $(*)$ of Lemma \ref{l:induction-and-restriction-commute}.
\end{lem}

\begin{proof}
Write $n=k+l$ and $A=C^\sim$, $B=D^\sim$. First note that the square above is a pullback diagram, i.e., $A=C\cap B$, which follows from the fact that $\Sigma_k \times \Sigma_l = \Sigma_n \cap (\cP_k \times \cP_l)$.

Define $X\subseteq B\times C$ as follows: $(b,c)\in X$ if and only if the partial bijection $\pi(b)$ is order-preserving on $\mathrm{im}(\pi(c))^\perp \coloneqq \{1,\ldots,n\} \smallsetminus \mathrm{im}(\pi(c))$. We need to show that (a) every $d\in D$ is of the form $bc$ for $(b,c)\in X$ and that (b) if $b_1 c_1 = b_2 c_2$ for $(b_i,c_i)\in X$ then $b_2 a=b_1$ and $ac_1=c_2$ for some $a\in A$.

(a). Given any $d\in D$, the partial bijection $\pi(d)$ will not in general preserve the partition $\{1,\ldots,l\} \sqcup \{l+1,\ldots,n\}$, but we may find some permutation $\sigma\in\Sigma_n$ such that $\sigma^{-1}\pi(d)$ does preserve it, i.e., lies in the submonoid $\cP_k \times \cP_l$. Moreover, it does not matter how $\sigma^{-1}$ acts away from the image of $\pi(d)$ so we may assume that it is order-preserving on $\mathrm{im}(\pi(d))^\perp$. We assumed that the restriction of $\pi$ to underlying groups is surjective, so we may pick $b\in B=D^\sim$ such that $\pi(b)=\sigma$. Now define $c=b^{-1}d$, so of course $bc=d$. Since $\pi(c)=\sigma^{-1}\pi(d) \in \cP_k \times \cP_l$ we know that $c\in C$. It remains to show that $(b,c)$ is in $X$, i.e., that $\pi(b)$ is order-preserving on $\mathrm{im}(\pi(c))^\perp$. But we ensured that $\sigma^{-1}$ is order-preserving on $\mathrm{im}(\pi(d))^\perp$, which is equivalent to saying that $\sigma$ is order-preserving on $\mathrm{im}(\sigma^{-1}\pi(d))^\perp$, which is precisely the required condition.

(b). Define $a=b_2^{-1}b_1 \in B$. It then immediately follows that $b_2 a=b_1$ and $ac_1 = c_2$, so we just have to show that $a\in A$. Since $A=C\cap B$ this means we just need to show that $a\in C$, in other words that $\pi(a)\in \cP_k \times \cP_l$ --- i.e.\ that $\pi(a)$ preserves the partition $\{1,\ldots,l\} \sqcup \{l+1,\ldots,n\}$.
First, since $\pi(a)\pi(c_1)=\pi(c_2)$ with $\pi(c_i)$ both preserving the partition, it follows that $\pi(a)$ restricted to $\mathrm{im}(\pi(c_1))$ preserves the partition. We will now show that $\pi(a)$ restricted to $\mathrm{im}(\pi(c_1))^\perp$ is order-preserving --- which will imply that $\pi(a)$ preserves the partition on all of $\{1,\ldots,n\}$.
By the definition of $X$, we know that $\pi(b_2)$ is order-preserving on $\mathrm{im}(\pi(c_2))^\perp$. Hence $\pi(b_2)^{-1}$ is order-preserving on
\[
\pi(b_2)\bigl( \mathrm{im}(\pi(c_2))^\perp \bigr) = \mathrm{im}(\pi(b_2 c_2))^\perp = \mathrm{im}(\pi(b_1 c_1))^\perp = \pi(b_1)\bigl( \mathrm{im}(\pi(c_1))^\perp \bigr).
\]
Combined with the fact that $\pi(b_1)$ is order-preserving on $\mathrm{im}(\pi(c_1))^\perp$ this tells us that $\pi(a) = \pi(b_2)^{-1} \pi(b_1)$ is order-preserving on $\mathrm{im}(\pi(c_1))^\perp$, as required.
This completes the proof of property (b) of $X\subseteq B\times C$, so the square of submonoids \eqref{eq:square-of-submonoids} satisfies condition $(*)$ of Lemma \ref{l:induction-and-restriction-commute}.
\end{proof}

\subsection{Returning to the definition of height.}\label{sss:height-general}

Following on from \S\ref{sss:height}, we give details of the alternative definition of the \emph{height} of a twisted coefficient system with indexing category $\cC \in \cati$. For the first step we define a subcategory $\cB \subseteq \cC$, a faithful functor $\cA \to \cB$ and an $\bN$-grading of the objects of $\cA$. For the second step, given a functor $T \colon \cC \to \ab$, we define the \emph{cross-effect functor} $T^\prime \colon \cA \to \ab$ associated to $T$, and show that $\mathrm{Ind}_{\cA \to \cB}(T^\prime)|_{\cB^{\sim}} \cong\, T|_{\cB^{\sim}}$, where $\cB^\sim$ denotes the underlying groupoid of $\cB$. As stated in \S\ref{sss:height}, the \emph{height} of $T$ is then the smallest $n$ such that $T^\prime$ is supported on the subcategory $\cA^{\leq n} \subseteq \cA$, in other words vanishes on the subcategory $\cA^{>n} \subseteq \cA$. In Remark \ref{rmk:two-defs-of-height-agree}, we explain why this agrees with the \emph{height} of $T$ as defined in \S\ref{para:specialise-this-paper}.

\paragraph*{The first step.}

Recall that $\cC\in\cati$ comes equipped with functors $s\colon\cI \to \cC$ and $\pi\colon \cC\to \Sigma$. The objects of $\cB$ are non-negative integers and those of $\cA$ are pairs of non-negative integers. Both are simply disjoint unions of monoids, i.e.\ they consist only of endomorphisms, so we just need to specify $\mathrm{End}_{\cA}(k,l)$ and $\mathrm{End}_{\cB}(n)$. As in \S\ref{sss:height} and in Lemma \ref{l:checking-property-star} above, let $\cP_n$ denote the monoid $\mathrm{End}_\Sigma(n)$ of partial bijections of $\{1,\ldots,n\}$ and write $l=n-k$ for convenience. There is a submonoid isomorphic to $\cP_k \times \cP_l$ consisting of those partial bijections that respect the partition $\{1,\ldots,l\}\sqcup\{l+1,\ldots,n\}$ wherever they are defined. We now define
\begin{align*}
\mathrm{End}_{\cB}(n) &= \mathrm{End}_\cC(s(n)) \\
\mathrm{End}_{\cA}(k,l) &= \text{preimage of } \cP_k \times \cP_l \text{ under the map } \pi \colon \mathrm{End}_{\cB}(n) \longrightarrow \mathrm{End}_\Sigma(n) = \cP_n.
\end{align*}
This completes the definitions of $\cB$ and $\cA$. The grading of the objects of $\cA$ is given by setting $\mathrm{deg}((k,l)) = k$. There is an obvious faithful functor $\cA \to \cB$, given on objects by $(k,l)\mapsto k+l$, and an embedding of categories $\cB \hookrightarrow \cC$.

\paragraph*{The second step.}

Recall that the monoid $\cI_n = \mathrm{End}_{\cI}(n)$, which is the submonoid of $\cP_n$ consisting of all idempotent elements, is isomorphic to the power set $\mathsf{P}(\{1,\ldots,n\})$, which is a commutative monoid via the operation $\cup$. The correspondence sends a subset $S\subseteq\{1,\ldots,n\}$ to the idempotent element $f_S\in\cP_n$ that is undefined on $S$ and the identity elsewhere.

Thus, given an object $\cC\in\cati$ and a functor $T\colon \cC\to\ab$, we have a collection of idempotents
\[
Ts(f_S) \colon T(s(n)) \longrightarrow T(s(n)) \qquad\text{for } S\subseteq \{1,\ldots,n\}.
\]
Write $l=n-k$ for convenience. In order to define the functor $T^\prime \colon \cA \to \ab$ we need to specify an $\mathrm{End}_{\cA}(k,l)$-module for each pair $(k,l)$ of non-negative integers. As an abelian group, we define it to be
\begin{equation}\label{eq:functorial-cross-effect}
T^\prime(k,l) \;=\; \mathrm{im}\bigl( Ts(f_{\{1,\ldots,l\}}) \bigr) \cap \bigcap_{i=l+1}^n \mathrm{ker} \bigl( Ts(f_{\{i\}}) \bigr) \quad\leq\quad T(s(n)).
\end{equation}
The monoid $\mathrm{End}_\cC(s(n))$ acts on $T(s(n))$ via the functor $T$, and it turns out (see $3$ lines below) that each element $\phi$ of its submonoid $\mathrm{End}_{\cA}(k,l)$ sends the subgroup $T^\prime(k,l)$ to itself. Hence $T^\prime(k,l)$ is an $\mathrm{End}_{\cA}(k,l)$-module --- and so we have defined the functor $T^\prime \colon \cA \to\ab$.

The claim in the previous paragraph follows from the fact that $\phi$ commutes with the element $s(f_{\{1,\ldots,l\}})$ and with the set of elements $\bigl\lbrace s(f_{\{l+1\}}),\ldots,s(f_{\{n\}})\bigr\rbrace$. This in turn follows from the fact that $\pi(\phi)\in \cP_k \times \cP_l$ together with the ``locality'' property \eqref{eq:locality} of $\cC \in \cati$.

It remains to show the following:

\begin{prop}
The functors $\mathrm{Ind}_{\cA \to \cB}(T^\prime)$ and $T$ are isomorphic on the subgroupoid $\cB^{\sim}$.
\end{prop}

\begin{proof}
Since $\cB$ is a disjoint union of monoids, this is just a more elaborate way of saying that for each $n\geq 0$ there is an isomorphism of modules over $\mathrm{Aut}_\cC(s(n)) = \mathrm{Aut}_{\cB}(n)$:
\begin{equation}\label{eq:cross-effect-decomp-generalised}
T(s(n)) \;\cong\; \mathrm{Ind}_{\cA \to \cB} (T^\prime)(n).
\end{equation}
The proof of Proposition 3.5 of \cite{Palmer2018Twistedhomologicalstability} generalises verbatim to this setting to show that the left-hand side of \eqref{eq:cross-effect-decomp-generalised} is isomorphic to
\[
\bigoplus_{k+l=n} \bigl( \bZ\mathrm{Aut}_{\cB}(n) \otimes_{\bZ\mathrm{Aut}_{\cA}(k,l)} T^\prime(k,l) \bigr).
\]
The categories $\cA$ and $\cB$ are both disjoint unions of monoids, so, as remarked in \S\ref{sss:special-cases}, the right-hand side of \eqref{eq:cross-effect-decomp-generalised} may be written as follows:
\[
\bigoplus_{k+l=n} \bigl( \bZ\mathrm{End}_{\cB}(n) \otimes_{\bZ\mathrm{End}_{\cA}(k,l)} T^\prime(k,l) \bigr).
\]
To finish the proof we will apply Lemma \ref{l:induction-and-restriction-commute}, so we need to know that for each $k+l=n\geq 0$ the square of submonoids
\begin{center}
\begin{tikzpicture}
[x=1mm,y=1mm]
\node (tl) at (0,10) {$\mathrm{End}_{\cA}(k,l)$};
\node (tr) at (40,10) {$\mathrm{End}_{\cB}(n)$};
\node (bl) at (0,0) {$\mathrm{Aut}_{\cA}(k,l)$};
\node (br) at (40,0) {$\mathrm{Aut}_{\cB}(n)$};
\incl{(tl)}{(tr)}
\incl{(bl)}{(br)}
\incl{(bl)}{(tl)}
\incl{(br)}{(tr)}
\end{tikzpicture}
\end{center}
satisfies condition $(*)$ of that lemma. This will be given by Lemma \ref{l:checking-property-star} as long as the homomorphism
\[
\pi\colon \mathrm{End}_{\cB}(n) = \mathrm{End}_\cC(s(n)) \longrightarrow \mathrm{End}_\Sigma(n) = \cP_n
\]
(as well as its restriction to maximal subgroups) is surjective. But this is true by definition for any $\cC\in\cati$ (see Definition \ref{def:cati}).
\end{proof}

\begin{rmk}\label{rmk:two-defs-of-height-agree}
Finally, we note that this description of $\text{height}(T \colon \cC \to \ab)$ for an object $\cC \in \cati$ agrees with the definition of $\text{height}(T \colon \cC \to \ab)$, given in \S\ref{para:specialise-this-paper}, for any category $\cC$ equipped with a functor $\cI \to \cC$ (such as any object of $\cati$). In other words, it is just a different way of packaging the same definition. To see this: the height of $T$, as defined in this section, is the largest $k$ such that \eqref{eq:functorial-cross-effect} is non-zero for some value of $l$, whereas the height of $T$, as defined in \S\ref{para:specialise-this-paper}, is the largest $n$ such that \eqref{eq:subobject-of-Tsm-2} is non-zero for some value of $m-n$. But the objects \eqref{eq:functorial-cross-effect} and \eqref{eq:subobject-of-Tsm-2} are the same, with $k \leftrightarrow n$ and $l \leftrightarrow m-n$.
\end{rmk}

\renewcommand{\thesection}{A}
\section[Appendix]{Whitney's embedding theorem for manifolds with collared boundary}

The Whitney Embedding Theorem implies that any (paracompact) smooth manifold without boundary admits an embedding into some Euclidean space. In footnote \ref{f:Whitney} on page \pageref{f:Whitney}, the analogous statement for manifolds with collared boundary was used. We could not find an explicit reference for this in the literature, so we explain here briefly how to deduce the statement for manifolds with collared boundary from the statement for manifolds without boundary.

\begin{lem}\label{lem:Whitney-with-boundary}
Any (paracompact) smooth manifold $M$ equipped with a collar neighbourhood admits a neat embedding into some Euclidean half-space $\bR^k_+ = \{ (s_1,\ldots,s_k) \in \bR^k \mid s_k \geq 0 \}$.
\end{lem}

A \emph{collar neighbourhood} means a proper embedding $c \colon \partial M \times [0,1] \hookrightarrow M$ such that $c(p,0)=p$. An embedding $f \colon M \hookrightarrow \bR^k_+$ is \emph{neat} if it takes $\partial M$ into $\bR^{k-1} = \partial (\bR^k_+)$ and the interior of $M$ into the interior of $\bR^k_+$ and, moreover, there is $\varepsilon > 0$ such that for all $(p,t) \in \partial M \times [0,\varepsilon)$ we have $f(c(p,t)) = (f(p),t)$.

\begin{proof}
First, we may embed $M$ into a manifold without boundary, either by gluing two copies of $M$ together along their common boundary or simply by attaching an open collar to the boundary of a single copy of $M$. By Whitney's Embedding Theorem we therefore obtain an embedding $g \colon M \hookrightarrow \bR^{k-1}$ for some $k$. Now choose a smooth embedding $(x,y) \colon [0,1] \hookrightarrow [0,1]^2$ such that
\begin{itemizeb}
\item for $0\leq t\leq \frac{1}{4}$ we have $x(t)=0$ and $y(t)=t$,
\item for $\frac{3}{4} \leq t\leq 1$ we have $x(t)=t$ and $y(t)=1$.
\end{itemizeb}
We may then define the required neat embedding $f \colon M \hookrightarrow \bR^k_+$ as follows. If $p \in M\smallsetminus\mathrm{image}(c)$ then $f(p) = (g(p),1)$. If $p \in \partial M$ and $t\in [0,1]$ then we define $f(c(p,t)) \;=\; (g(c(p,x(t))),y(t))$.
\end{proof}

The idea is that most of $M$ -- the part far away from its boundary -- is embedded into the affine hyperplane $\bR^{k-1} \times \{1\}$, and its collar neighbourhood is bent smoothly downwards towards the linear hyperplane $\bR^{k-1} \times \{0\}$, using the functions $x$ and $y$, such that the boundary of $M$ is on this hyperplane and the part of the collar neighbourhood closest to the boundary of $M$ is embedded so that it rises vertically upwards from the hyperplane.


\phantomsection
\addcontentsline{toc}{section}{References}
\renewcommand{\bibfont}{\normalfont\small}
\setlength{\bibitemsep}{0pt}
\printbibliography

\end{document}